\documentclass[12pt]{extarticle}
\usepackage{amsmath, amsthm, amssymb, color}
\usepackage[colorlinks=true,linkcolor=blue,urlcolor=blue]{hyperref}
\usepackage{graphicx}
\usepackage{caption}
\usepackage{mathtools}
\usepackage{enumerate}
\usepackage{verbatim}
\usepackage{tikz,tikz-cd,tikz-3dplot}
\usepackage{amssymb}
\usetikzlibrary{matrix}
\usetikzlibrary{arrows}
\usetikzlibrary{shapes.geometric}
\usepackage{algorithm}
\usepackage[noend]{algpseudocode}
\usepackage{caption}
\usepackage[normalem]{ulem}
\usepackage{subcaption}
\usepackage{multicol}
\usepackage[T1]{fontenc}
\oddsidemargin \evensidemargin
\usepackage{makecell}
\usepackage{array}
\usepackage{enumitem}
\usepackage{booktabs}
\usepackage{diagbox}

\newtheorem{theorem}{Theorem}[section]

\newtheorem{proposition}[theorem]{Proposition}
\newtheorem{lemma}[theorem]{Lemma}
\newtheorem{corollary}[theorem]{Corollary}

\theoremstyle{definition}
\newtheorem{definition}[theorem]{Definition}

\newtheorem{remark}[theorem]{Remark}
\newtheorem{conjecture}[theorem]{Conjecture}

\newtheorem{example}[theorem]{Example}

\newcommand{\PP}{\mathbb{P}}
\newcommand{\RR}{\mathbb{R}}

\newcommand{\CC}{\mathbb{C}}

\newcommand{\ostar}{%
  \tikz[baseline=-0.6ex]%
    \node[draw,star,star points=5,star point ratio=2.25,minimum size=1.6ex,inner sep=0pt
    ] {};%
}

\title{\bf Distance Optimization in \\ the 
Grassmannian of Lines}
\author{Hannah Friedman, Andrea Rosana and Bernd Sturmfels}

\date{}

\begin{document}

\maketitle

\begin{abstract} \noindent
The square of a skew-symmetric matrix is a symmetric matrix
whose eigenvalues have even multiplicities.
When the matrices have rank two, they represent the Grassmannian
of lines, and the squaring operation takes Pl\"ucker coordinates to projection coordinates.
We develop metric algebraic geometry for
varieties of lines in this linear algebra setting.
The Grassmann distance (GD) degree is introduced as
 a new invariant for subvarieties of a Grassmannian.
We study the GD degree for Schubert varieties and other models.
\end{abstract}

\section{Introduction}
\label{sec1}

The Grassmannian ${\rm Gr}(2,n)$ is the variety of all lines in projective $(n-1)$-space $\PP^{n-1}$.
It is embedded in the projective space
in $\PP^{\binom{n}{2}-1}$ as the variety of skew-symmetric $n \times n$-matrices
$X= [x_{ij}]$ of rank $2$. The matrix entries $x_{ij}$ are the
Pl\"ucker coordinates.
The homogeneous prime ideal of ${\rm Gr}(2,n)$ is generated by the
$4 \times 4$ Pfaffians of our skew-symmetric matrix:
\begin{equation}
\label{eq:pluckerquadric} 
{\rm Pf}\begin{small}
 \begin{bmatrix}
\,\,0 & \,\, x_{ij} & \,\,x_{ik} & \,\,x_{il}\,\, \\
\,-x_{ij} & 0 & \,\,x_{jk} & \,\,x_{jl} \,\,\\
\,-x_{ik} & \!\!-x_{jk} & 0 & \,\,x_{kl}\, \,\\
\,-x_{il} & \!\!-x_{jl} & \!\!-x_{kl} & 0 \,
\end{bmatrix}
\end{small} \,\, = \,\,
 x_{ij} x_{kl} \,- \, x_{ik} x_{jl} \,+ \, x_{il} x_{jk} \quad\,
 {\rm for} \,\,\, 1 \!\leq\! i \!<\! j\! < \! k \!< \! l  \!  \leq\! n.
\end{equation}
While Pl\"ucker coordinates are natural for classical algebraic geometry,
in applications over the real numbers $\RR$,
one often works with the square $P = X^2$; see \cite{DFRS, LLY, LY}.
The symmetric $n \times n$ matrix $P = [p_{ij}]$ has rank $2$,
and  its trace equals the squared Frobenius norm of $X$:
$$  {\rm trace}(P) \, = \,
p_{11} + \cdots + p_{nn} \,\, = \,\, 2(x_{12}^2 + x_{13}^2 + \cdots + x_{n-1,n}^2) \,=\, ||X||^2.$$
By \cite[Corollary 2.5]{DFRS},
the orthogonal projection $\tilde P : \RR^n \to {\rm image}(X) \simeq \RR^2$ is given~by
\begin{equation}
\label{eq:projmatrix} \tilde P \,\,\coloneq \,\,\frac{2}{{\rm trace}(P)} P \,\, = \,\, \frac{2}{||X||^2} X^2.
\end{equation}

We here view the Grassmannian through the lens
of metric algebraic geometry \cite{BKS}.
Given two points in ${\rm Gr}(2,n)$, encoded by
projection matrices $\tilde P = [\tilde p_{ij}]$ and $\tilde Q = [\tilde q_{ij}]$ as  in~(\ref{eq:projmatrix}),
we measure their distance in the Frobenius norm on the space
of symmetric $n \times n$ matrices:
\begin{equation}
\label{eq:frobenius}
 || \tilde P - \tilde Q || \,\, = \,\,
\biggl(\sum_{i=1}^n (\tilde p_{ii} - \tilde q_{ii})^2 \,+\,
2 \sum_{i < j}  (\tilde p_{ij} - \tilde q_{ij})^2 \biggr)^{1/2} .
\end{equation}

We consider the Euclidean distance problem
for a subvariety $\mathcal{M}$ in ${\rm Gr}(2,n)$.
This is  reminiscent of \cite[Chapter 11]{BKS}.
The model $\mathcal{M}$ is a family of lines in $\PP^{n-1}$.
We are given a data point $Q$ which is also a line
in $\PP^{n-1}$.
Our task is to compute a line $P$ in the model $\mathcal{M}$
which best approximates the given line $Q$.
Thus we wish to solve the optimization problem
\begin{equation}
\label{eq:optproblem}
 {\rm Minimize} \,\, || \tilde P - \tilde Q ||^2 \,\,\,\,\hbox{subject to} \,\,
\tilde P \in \mathcal{M}. 
\end{equation}
This is a polynomial optimization problem, so we can solve it
using methods from computational algebraic geometry.
The key parameter is the {\em Grassmann distance degree} (GD degree)
of $\mathcal{M}$. We define this as the
number of complex critical points of (\ref{eq:optproblem})
for generic $ Q \in {\rm Gr}(2,n)$.

The GD degree is
analogous to the ML degree \cite{BKS} and the ED degree \cite{DHOST}.
It provides an algebraic alternative to the 
Grassmann distance complexity due to Lerario and Rosana \cite{LR}.
Their work rests on the {\em geodesic distance}, which is not
an algebraic function \cite[Section~3.6]{LR}.

Conway, Hardin and Sloane \cite{CHS}
refer to  our metric (\ref{eq:frobenius}) as the {\em chordal distance}, and they report:
{\em ``We have made extensive computations on $\ldots$ optimal packings.
These computations have led us to conclude that the best definition
of distance on Grassmannian space is the chordal distance.''}
We note that the chordal distance between two lines in $\PP^{n-1}$ is equal~to
\begin{equation}
\label{eq:frobenius2}
 || \tilde P - \tilde Q||^2 \,=\,
{\rm trace}\bigl( \,(\tilde P - \tilde Q)^2\, \bigr)
\,=\, 4 - 2\cdot {\rm trace}(\tilde P \tilde Q) \, = \, 4 - 2 \lambda - 2 \mu,
\end{equation}
where $\lambda \geq \mu$ are the non-zero eigenvalues of 
the product $\tilde P \tilde Q$ of the two projection matrices.

Many metrics in Grassmannian optimization  are functions of the eigenvalues $\lambda$
and $\mu$; see \cite[Theorem 2]{YL}. The following table, derived from  \cite[Table 2]{YL}, 
shows nine options:
\begin{small}
\begin{center}
  \begin{tabular}{|c|c|c|}
    \toprule
  & Distance squared in $\lambda, \mu$ & Minimizer \\
    \midrule
  Chordal & $4 - 2\lambda - 2\mu $ & $\ostar$ \\
  Geodesic & $\arccos(\sqrt{\lambda})^2 + \arccos(\sqrt{\mu})^2$ & $\circ$\\
  Procrustes & $4 - 2\sqrt{\lambda} - 2\sqrt{\mu}$ & $\diamond$\\
  \midrule
  Binet-Cauchy & $1 - \lambda\mu$ & $\scriptstyle \triangle$\\
  Fubini-Study & $\arccos(\sqrt{\lambda\mu})^2$ & $\scriptstyle \triangle$\\
  Martin & $-\log(\lambda\mu) $ & $\scriptstyle \triangle$\\
  \midrule
  Asimov & $\arccos (\sqrt{\mu})^2$ & $\square$\\
  Projection & $1 - \mu$ &  $\square$\\
  Spectral & $2 - 2\sqrt{\mu}$ &  $\square$\\
  \bottomrule
  \end{tabular}
  \end{center}
  \end{small}
While our focus is on the chordal distance (\ref{eq:frobenius})-(\ref{eq:frobenius2}), 
the algebraic tools introduced in this paper are applicable to any of the metrics above.
The following example demonstrates this.

\begin{example}[Schubert surface] \label{ex:spectralcurve}
Let $n = 5$. For our model, we take the surface
$$\mathcal S_{25} \,\,=\,\, \{X \in {\rm Gr}(2,5)\,:\, x_{12} = x_{13} = x_{14} = x_{15} = x_{23} = x_{24}  = x_{34} \,=\, 0\}.$$
We write the model as a  projection matrix $\tilde P$, and we  fix a projection matrix $\tilde Q$
for the data:
$$
  \!\tilde P =\begin{small} \frac{1}{x_{25}^2\!+\!x_{35}^2 \!+\!  x_{45}^2}
  \begin{bmatrix}
  \,0 & 0 & 0 & 0 & 0 \\
\,  0 & x_{25}^2 & \! x_{25} x_{35}\! & \!\!x_{25} x_{45}\! &  0     \\
\,  0 & \!x_{25} x_{35} \!& x_{35}^2 &\! \!x_{35} x_{45} \!& 0      \\
\,  0 & \!x_{25} x_{45} \!& \!x_{35} x_{45}\! & x_{45}^2 & 0       \\
  \, 0 & 0 &     0  &    0  &  \!\! \! x_{25}^2\!+\!x_{35}^2\!+\!x_{45}^2 
   \end{bmatrix} \end{small},\,\,\,
    \tilde Q =\begin{small} \frac{1}{10}
    \begin{bmatrix}
      \,\, 6 & 4 & 2 & 0 &\!\!\! -2\\
      \,\,4 & 3 & 2 & 1 & 0\\
      \,\,2 & 2 & 2 & 2 & 2\\
      \,\,0 & 1 & 2 & 3 & 4\\
     \,\, -2 & 0 & 2 & 4 & 6
    \end{bmatrix}\!.\end{small}
$$
We define the  {\em spectral region} of the pair $(\mathcal{M},\tilde Q)$ as the
set of all pairs $\lambda \geq \mu$ of eigenvalues of 
the matrices $\tilde P \tilde Q$, where $\tilde Q$ is fixed and
$\tilde P$ ranges over the real points of the model $\mathcal{M}$.

This spectral region is shown in Figure \ref{fig:spectralregion}.
Its lower left corner is $(3/5,0)$. The eigenvalues are bounded above by
$\lambda \leq 1$ and $ \mu \leq 2/5$. The nonlinear boundary curve is the quartic
  \begin{equation} \label{eq:spectralcurve}
    100\lambda^2\mu^2 - 100\lambda^2\mu + 49\lambda^2 - 100\lambda \mu^2 + 170\lambda \mu - 84\lambda + 49\mu^2 - 84\mu + 36 \,\,=\,\, 0. 
  \end{equation}
  This specifies the {\em Pareto frontier}.
  The map from the surface $\mathcal{M}$ onto the $(\lambda,\mu)$-plane is $4$-to-$1$.
 The nine metrics give rise to five distinct minimizers on the Pareto frontier.
 For instance, the circle is for the geodesic distance in \cite{LR}.
 The minimizer for our distance (\ref{eq:frobenius})-(\ref{eq:frobenius2}),
 indicated by a star, satisfies $x_{25} = x_{35} = x_{45} = 1$, and
 $\lambda = \frac{3 + \sqrt{3}}{5} = 0.9464$,
    $\mu = \frac{3-\sqrt{3}}{5} =         0.2536$, with
    optimal value $8/5$.
    The curve (\ref{eq:spectralcurve}) can be viewed as the
    universal solution for all nine metrics.
    \end{example}

\begin{figure}[h]
  \center
  \includegraphics[scale=0.55]{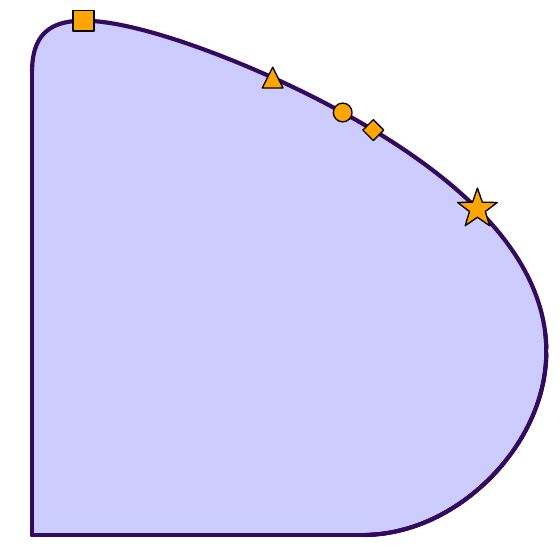}
  \caption{The spectral region comprises all eigenvalue pairs $(\lambda,\mu)$
  as $X$ ranges over $\mathcal{M}$.
  The marked points are the optimal solutions for the nine metrics 
  on ${\rm Gr}(2,5)$ in the table~above.
  }\label{fig:spectralregion}
\end{figure}

We now discuss the organization and contributions of this paper.
In Section \ref{sec2} we study the squaring 
map $X \mapsto X^2$ for skew-symmetric $n \times n$ matrices $X$.
Theorem \ref{thm:skew-sym} characterizes its image and fibers.
The restriction to the rank $2$ locus identifies ${\rm Gr}(2,n)$
with the projection Grassmannian ${\rm pGr}(2,n)$, as shown in \cite[Section 2]{DFRS}.
Theorems \ref{thm:baselocus} and \ref{thm:stenzel} feature
varieties that arise naturally from multiplying skew-symmetric
 matrices.
In Section \ref{sec3} we introduce the polynomial ideals
that characterize the critical points of (\ref{eq:optproblem}),
 where $\tilde Q$
is an arbitrary symmetric $n \times n$ matrix.
Theorem \ref{thm:EDgeneral} identifies what is special about this ED problem.
Theorem \ref{thm:rtlfct} counts the critical points of a rational function
and is of independent interest.

In Section \ref{sec4} we turn to the GD problem.
Now the data point $\tilde Q $ itself lies in  $ {\rm pGr}(2,n)$.   
The critical ideal and the GD degree are introduced in
Definition \ref{def:GDdefs}. Their computation requires the
removal of the algebraic cut locus.
Theorem \ref{thm:excess}
explains the difference between the ED degree and the GD degree.
 Section \ref{sec5} is devoted to lines in $3$-space ($n=4$).
 Models $\mathcal{M}$ in ${\rm Gr}(2,4)$ have dimension
one (ruled surfaces), two (congruences) and three (complexes).
They  appear in  computer vision \cite{PST}, 
physics~\cite{MHPS}, and many other applications.  
Degree formulas for   generic complete intersection are presented in
Theorem \ref{thm:CI} and Conjecture \ref{conj:surface}.

In Section \ref{sec6} we discuss our  optimization problem
for  Schubert varieties in~${\rm Gr}(2,n)$.  
Conjectures \ref{conj:schubert} and \ref{conj:factor4}
invite our readers to study these
 varieties in projection coordinates.
Formulas for their
ED and GD degrees are given in
Theorem \ref{thm:schubertGD1}
and Conjectures~\ref{conj:schubertED}~and~\ref{conj:schubertGD}.

\section{Squaring Skew-Symmetric Matrices}
\label{sec2}

In this section, we develop foundational linear algebra for the metric algebraic
geometry of the Grassmannian ${\rm Gr}(2,n)$. A point in ${\rm Gr}(2,n)$
is the row span of a $2 \times n$ matrix $A = [a_{ij}]$.
A less ambiguous represention is given by the Pl\"ucker coordinates, which are the entries of
\begin{equation}
\label{eq:lessambi} X = [x_{ij}] \,\, = \,\, A^T 
\begin{small} \begin{bmatrix} \,\,\,\,0 & 1\,\,\, \\ -1 & 0\,\,\, \end{bmatrix} \end{small} A. 
\end{equation}
This is a skew-symmetric $n \times n$ matrix of rank $2$.
Following \cite{DFRS},
the second life of ${\rm Gr}(2,n)$ is the
projection Grassmannian ${\rm pGr}(2,n)$, which consists of 
symmetric  $n \times n$  matrices (\ref{eq:projmatrix}) of rank~$2$.
In what follows we examine this passage from the first life to the second life
from a more general perspective. Namely, we temporarily drop the
hypothesis that $X$ has rank~two.

Taking the square is a rational map
from skew-symmetric matrices to symmetric matrices:
\begin{equation}\label{eq:square}
(-)^2 \,:\, \PP(\wedge^2 \CC^n)\, \dashrightarrow \,\PP({\rm Sym}^2 \CC^n),  \quad\quad X \mapsto X^2.
\end{equation}
This map is rational but  it is not a morphism. Its base locus contains complex matrices like
\begin{equation}\label{eq:base-matrix} \qquad
X \,=\,  \begin{small} \begin{bmatrix}
  0 & 1 & i & 0\,\\
  -1 & 0 & 0 & i \,\\
  -i & 0 & 0 & \!\!-1\,\,\\
  0 & -i & 1 & 0\,
  \end{bmatrix} \end{small} \,\, \in \, {\rm Gr}(2,4) \,\subset \, \PP(\wedge^2 \CC^4).
\end{equation}
We write  $\mathcal{V}_{X^2}$ for the Zariski closure in
 $\PP({\rm Sym}^2 \CC^n)$ of the image of the squaring map $(-)^2$.

\begin{theorem}\label{thm:skew-sym}
  The variety $\mathcal{V}_{X^2}$ has
the  expected dimension $\binom{n}{2}- 1$. It is
   the Zariski closure of the set of symmetric diagonalizable
  matrices whose nonzero eigenvalues have multiplicity~$2$.
  The fiber of $\,(-)^2$ containing
a generic rank $2r$ skew-symmetric $n \times n$ matrix $X$
 has size~$2^{r-1}$. 
\end{theorem}

The proof rests on the conjugation action of the orthogonal group ${\rm O}(n)$ 
on the space~$\wedge^2 \CC^n$.
Normal forms for the orbits 
are given in \cite{DRZ, gantmacher}.
The eigenvalues can be read off of the normal form in
Gantmacher's book \cite[\S XI.4]{gantmacher}. We prefer it for that reason. 
The nonzero eigenvalues of skew-symmetric matrices come in pairs $\pm \lambda$.
This is clear from the elementary divisors.

\begin{lemma}[{\cite[Theorem 7]{gantmacher}}]
  If $\lambda \neq 0$ and $(z - \lambda)^p$ is an elementary divisor of a skew-symmetric matrix $X$, then so is $(z + \lambda)^p$. Moreover, even powers of $z$ come in pairs. 
\end{lemma}

The normal form uses $2p \times 2p$ blocks for the pair of elementary divisors $(z \pm \lambda)^p$, and $q\times q$ matrices for elementary divisors $z^q$ with $q$ odd. 
A generic rank $2r$ skew-symmetric $n \times n$ matrix $X $ has
$2r + 1$ distinct eigenvalues $\pm \lambda_1i, \ldots, \pm \lambda_ri, 0$ or, if  $n = 2r$, it has $2r$ eigenvalues $\pm \lambda_1i, \ldots, \pm \lambda_ri$.
Its normal form
is a direct sum of $2 \times 2$ blocks and $1 \times 1$ zero-blocks: 
\begin{equation}
\label{eq:twobytwo} X \,\,= \,\,\begin{bmatrix} 0 & \lambda_1 \, \\ - \lambda_1 & 0 \, \end{bmatrix} \,\oplus \,
 \begin{bmatrix} 0 & \lambda_2 \, \\ - \lambda_2 & 0 \, \end{bmatrix} \,\oplus \cdots \oplus
 \, \begin{bmatrix} 0 & \lambda_r \, \\ - \lambda_r & 0 \, \end{bmatrix} 
\, \oplus \, [0]\, \oplus \, \cdots \, \oplus [0]. 
 \end{equation}
The example in \eqref{eq:base-matrix} is a skew-symmetric matrix that cannot be put in the form \eqref{eq:twobytwo}.
In fact, it is already in  Gantmacher's normal form.  Compare this with the
Jordan canonical form.

\begin{proof}[Proof of Theorem \ref{thm:skew-sym}]
The squaring map commutes with the ${\rm O}(n)$-action.
Given a  generic point of rank $2r$ in the domain, we consider its normal form
(\ref{eq:twobytwo}). Its image under \eqref{eq:square} is
\begin{equation*}
\quad - X^2 \,\,= \,\, 
\begin{bmatrix} \lambda_1^2 & 0 \, \\ 0  & \lambda_1^2\, \end{bmatrix} \,\oplus \,
\begin{bmatrix} \lambda_2^2 & 0 \, \\ 0  & \lambda_2^2\, \end{bmatrix} \,\oplus \,\cdots \,\oplus \,
\begin{bmatrix} \lambda_r^2 & 0 \, \\ 0  & \lambda_r^2\, \end{bmatrix} \, \oplus \, [0]\, \oplus \, \cdots \, \oplus [0].
\end{equation*}
Hence $P = X^2$ has $r + 1$  distinct eigenvalues (or $r$ if $n = 2r$),
and all nonzero eigenvalues have multiplicity $2$.
The fiber over $P$ is given by the $2^r$ matrices which are obtained by
multiplying each $\lambda_i$ with $+1$ or $-1$. This gives rise to
$2^{r-1}$ elements $X$ in $\PP( \wedge^2 \CC^n)$.

Since the set of full rank skew-symmetric matrices with normal form \eqref{eq:twobytwo} is Zariski dense in the
domain $\PP(\wedge^2 \mathbb C^n)$, the image of this set 
is Zariski dense in the image variety $\mathcal{V}_{X^2}$.
\end{proof}

Over the real numbers $\RR$, every symmetric $n \times n$ matrix is diagonalizable.
It is equivalent to a diagonal matrix under the ${\rm O}(n)$-action. This is not true
over the complex numbers $\CC$. It thus follows from Theorem \ref{thm:skew-sym}
that the complex variety $\mathcal{V}_{X^2}$ is 
the Zariski closure of the set of real symmetric matrices whose nonzero eigenvalues come in pairs.
The varieties in $\PP({\rm Sym}^2 \RR^n) $ given by prescribed eigenvalue multiplicities 
were studied  by Weinstein in \cite{weinstein}. 

\begin{corollary}
  The image of the real locus $\PP(\wedge^2\mathbb R^n)$ under the  map 
(\ref{eq:square}) 
  consists~of negative semidefinite matrices $P \in {\rm Sym}^2 \RR^n$
    whose nonzero eigenvalues have even multiplicity. 
    If such a matrix $P$ has maximal rank then every complex matrix $X \in \wedge^2 \CC^n$ with $X^2 = P$ is~real.
\end{corollary}

\begin{proof}
  The first claim holds because every real skew-symmetric matrix has normal form \eqref{eq:twobytwo}.
  Consider a real symmetric matrix $P$ of maximal rank in the image of the squaring map.
  All elementary divisors of $P$ are linear.
  If $X^2 = P$, then the nonzero elementary divisors of $X$ are linear as well.
  Since $P$ has at most one $0$ eigenvalue, $X$ has normal form \eqref{eq:twobytwo}. 
  Furthermore, since $P$ is negative semidefinite, the eigenvalues of $X$ are of the form $\pm \lambda_ji$ with $\lambda_j$ real, so $X$ has real normal form.
  The matrix $X$ is obtained from the normal form by conjugating with a real orthogonal matrix that diagonalizes $P$. 
  Hence every matrix in the fiber is real. 
\end{proof}

We now discuss the ideal of the image variety $\mathcal{V}_{X^2}$.
The first non-trivial case is $n=3$.

\begin{example}[$n=3$] 
Here (\ref{eq:square}) is the Veronese embedding of $\PP^2$ into $\PP^5$. 
Thus $\mathcal{V}_{X^2}$ is a surface of degree $4$. Its prime ideal is generated by the
$2 \times 2$ minors of the symmetric~matrix
\begin{equation*}
2 P - {\rm trace}(P) \cdot {\rm Id}_3 \,\,\, = \,\,\,
  2\begin{bmatrix}
  p_{11} & p_{12} & p_{13}\\
  p_{12} & p_{22} & p_{23}\\
  p_{13} & p_{23} & p_{33}
  \end{bmatrix}
  - \,(p_{11} + p_{22} + p_{33}) \begin{bmatrix}
    1 & 0 & 0 \\ 0 & 1 & 0 \\ 0 & 0 & 1
  \end{bmatrix}.
  \end{equation*}
\end{example}

\begin{example}[$n=4$] 
The variety $\mathcal{V}_{X^2}$ has dimension $5$ and degree $6$ in the $\PP^9$
of symmetric $4 \times 4$ matrices $P = [p_{ij}]$. 
Its prime ideal is generated by the $2 \times 2$ minors of the $3 \times 3$ matrix
\begin{equation*}
\begin{bmatrix}
 \,p_{11}-p_{22}-p_{33}+p_{44} &  2 p_{13} + 2 p_{24} & 2 p_{12} - 2 p_{34} \\        
  2 p_{13} - 2 p_{24} & -p_{11}-p_{22}+p_{33}+p_{44} &  2 p_{14} + 2 p_{23}  \\
  2 p_{12} + 2 p_{34}  & -2 p_{14} + 2 p_{23} & - p_{11} + p_{22} - p_{33} + p_{44}\,\,
\end{bmatrix}\!.
\end{equation*}
Hence $\mathcal{V}_{X^2}$ is a cone over the Segre variety $\PP^2 \times \PP^2$.
Its apex is the $ 4 \times 4$ identity matrix.
\end{example}
    
\begin{example}[$n = 5,6,7$]   
The variety $\mathcal{V}_{X^2}$ has dimension $9$ and degree $46$ in the $\PP^{14}$ of symmetric $5 \times 5$ matrices.
Its ideal is minimally generated by $15$ cubics.
For $n=6$, the degree is $92$, and 
its ideal contains $20$ cubics. 
For $n=7$, it has dimension $20$ and degree~$1072$.
\end{example}

\begin{conjecture}
The prime ideal of the variety $\mathcal{V}_{X^2}$ is generated in
degree $\lceil n/2 \rceil$.
\end{conjecture}

The variety of primary interest in this paper is the
Grassmannian ${\rm Gr}(2,n)$, which is the
rank $2$ locus in $\PP(\wedge^2 \CC^n)$.
The rank $2r$ locus is the secant variety
${\rm Sec}_{r}({\rm Gr}(2,n))$. We know from 
Theorem \ref{thm:skew-sym}
that the degree of the map (\ref{eq:square}) is $2^{r-1}$ on this secant variety.
For $r=1$, the image of ${\rm Gr}(2,n)$ is the projective closure of the affine variety
${\rm pGr}(2,n)$ from \cite[Section 5]{DFRS}.
We denote this by $\overline{\rm p}{\rm Gr}(2,n)\subset \PP({\rm Sym}^2\CC^n)$
and we also call it the {\em projection Grassmannian}.

\begin{remark}
The variety
$\overline{\rm p}{\rm Gr}(2,n)$  has dimension
$2n-4$ and degree $2 \binom{2n-4}{n-2}$; see \cite{LY}.
Its prime ideal has $\binom{n+1}{2}$ quadratic generators, namely the entries of
$\,2 P^2  - {\rm trace}(P)\cdot P$. 
For $n \geq 6$, we also need the  $\,3 \times 3$ minors of~$P$.
A general statement appears in Conjecture~\ref{conj:schubert}.
\end{remark}

The singular locus of  $\overline{\rm p}{\rm Gr}(2,n)$  has dimension $n-2$
and degree $2^{n-1}$. It consists of all
traceless symmetric matrices of rank $1$.
This is a complex variety with no real points, namely it is the
image of the Fermat quadric
$V(x_1^2+x_2^2 + \cdots + x_n^2)$ under
the second Veronese~map. This singular locus is the Zariski closure of the image under (\ref{eq:square})
of the ${\rm O}(n)$-orbit of \eqref{eq:nilpotent}.

The map most relevant for our application to distance optimization is the composition
\begin{equation}
\label{eq:composition}
      \PP(\wedge^2\CC^n) \,\dashrightarrow\, \PP({\rm Sym}^2\CC^n)
\,      \dashrightarrow \,{\rm Sym}^2\CC^n, \quad\,
      X \,\mapsto\, X^2 = P, \quad  P \,\mapsto\, \frac{2}{{\rm trace}(P)} P = \tilde P.
    \end{equation}
    The second map is the identity on the affine open chart defined by ${\rm trace}(P) \not =0$.
The base locus of the composition 
 consists of all skew-symmetric matrices $X$ such that ${\rm trace}(X^2) = 0$.
  This is a hypersurface in $\PP(\wedge^2\CC^n)$ which has no real points.
For a complex point see \eqref{eq:nilpotent}.
The first map is (\ref{eq:square}). Its base locus is the set of
 skew-symmetric matrices $X$ with $X^2 = 0$.
  
\begin{theorem} \label{thm:baselocus}
  The ${\rm O}(n)$-orbit of a matrix $B$ with $k$ blocks \eqref{eq:base-matrix} and $r = n - 4k$ blocks 
  $[\,0\,]$ has codimension $4k^2 + 2kr + \binom{r}{2}$.
  The base locus of \eqref{eq:square} is the union of all such orbits. 
  It has two irreducible
  components if $n$ is a multiple of $4$ and is irreducible otherwise.
\end{theorem}

\begin{proof}
  The codimension of the orbit of $B$ is equal to the dimension of the stabilizer $\{R \in {\rm SO}(n) : R^T BR = B\}$.
   We compute the dimension of the corresponding Lie algebra, which is
  $\{X \in \wedge^2 \CC^n: [X,B] = 0\}$.
  Let $A$ be the direct sum of $k$ copies of \eqref{eq:base-matrix}. 
  We write skew-symmetric $n \times n$ matrices~as
  $X =     \begin{small} \begin{bmatrix}
      Y & Z\\
      Z^T & W
    \end{bmatrix} \end{small} $
  where $Y$ and $W$ are skew-symmetric and $Z$ is any $4k \times r$  matrix.
  The condition $[X, B] = 0$ is equivalent to $[A, Y]= 0$ and $AZ = 0$.
  We have $\binom r 2$ degrees of freedom to choose $W$.
  Since \eqref{eq:base-matrix} has rank 2, $A$ has nullity $2k$, so we have $2kr$ degrees of freedom to choose $Z$.
  We partition $Y$ into $4 \times 4$ blocks that commute with \eqref{eq:base-matrix}, with skew-symmetric blocks on the diagonal and general blocks on the off-diagonal.
  The space of skew-symmetric and general matrices that commute with \eqref{eq:base-matrix} have dimensions $4$ and $8$ respectively. Thus we have $4k + 8\binom{k}{2}=4k^2$ degrees of freedom to choose $Y$.
  
  The only block in the normal form  \cite[\S XI.4]{gantmacher} squaring to zero is \eqref{eq:base-matrix}; this proves the second claim.
   The closure ordering of nilpotent orbits \cite[Theorem~6.2.5]{CM} implies the last claim.
\end{proof}

\begin{example}\label{ex:baselocus}
Theorem \ref{thm:baselocus} shows
that the base locus of \eqref{eq:square} is empty for $n = 2, 3$.
For $n = 4$, it  is the Zariski closure of the ${\rm O}(4)$-orbit of \eqref{eq:base-matrix}. 
This has two components. Its ideal~is
\begin{small}
\begin{align*}
  \langle x_{14} - x_{23}, x_{13} + x_{24}, x_{12} - x_{34}, x_{23}^2 + x_{24}^2 + x_{34}^2 \rangle \cap
  \langle x_{14} + x_{23}, x_{13} - x_{24}, x_{12} + x_{34}, x_{23}^2 + x_{24}^2 + x_{34}^2 \rangle .
\end{align*}
\end{small}
For $n = 5,6,7$, the base locus  is irreducible. It is the orbit closure of
 \eqref{eq:base-matrix} padded with zeros.
 \end{example}

We are especially interested in the restriction of the base locus to
matrices $X$ of rank $2$.

\begin{corollary} \label{cor:baselocus}
  The base locus of \eqref{eq:square} on
  the Grassmannian ${\rm Gr}(2,n)$  is the  ${\rm O}(n)$-orbit closure
  of the matrix  \eqref{eq:base-matrix} padded with zeros. This is
  irreducible of dimension $2n-7$ for~$n \geq 5$.
   \end{corollary}

\begin{proof}
This follows from Theorem \ref{thm:baselocus}
by restricting to $k=1$ and $r=n-4$.
\end{proof}
   
\begin{remark}
The base locus has codimension $3$ in ${\rm Gr}(2,n)$. Its prime ideal is
minimally generated by the $\binom{n}{4}$ Pl\"ucker quadrics
and the $\binom{n+1}{2}$ entries of $X^2$. Computations 
show that its degree 
equals $\frac{4}{n-2}\binom{2n - 6}{n-3}$
for $n=4,5,\ldots,14$. The first values are $4,8,20,56,168,528,1716$.
\end{remark}

A natural extension of (\ref{eq:square}) is the map which multiplies two skew-symmetric matrices
$X$ and $Y$.
This operation was studied in 1922 by Stenzel  \cite{Ste}.
This was followed up in the linear algebra literature throughout the 20th century.
We refer to the bibliography in~\cite{IF}.
We define $\mathcal{V}_{XY}$,
$\mathcal{V}_{X^2Y^2}$ and $\mathcal{V}_{XY^2X}$
to be the varieties respectively parametrized by all products
$XY$, $X^2Y^2$ and $XY^2X$, where $X,Y \in  \PP(\wedge^2 \CC^n)$.
The first two varieties lie in the space $\PP(\CC^{n \times n})$
of all $n \times n$ matrices, while $\mathcal{V}_{XY^2X} \subset \PP({\rm Sym}^2 \CC^n)$.
 We call $\mathcal{V}_{XY}$ the
{\em Stenzel variety}.
For any of these three varieties 
$\mathcal{V}_{\bullet}$
we write $\mathcal{V}_{\bullet,\leq 2}$ for the
subvariety consisting of matrices of rank at most two.
These are obtained by taking $X$ and $Y$ from
the Grassmannian ${\rm Gr}(2,n)$.

\begin{theorem} \label{thm:stenzel}
  The Stenzel variety $\mathcal V_{XY}$ is the Zariski closure of  diagonalizable $n \times n$ matrices whose nonzero eigenvalues have multiplicity $2$. 
  The varieties $\mathcal{V}_{XY}$ and
  $\mathcal{V}_{X^2 Y^2}$ have dimension $n(n-1) - \lfloor n/2 \rfloor -1$. 
  The varieties  $\mathcal{V}_{XY, \leq 2}$ and
  $\mathcal{V}_{X^2 Y^2,\leq 2}$ have dimension $4n-8$. 
\end{theorem}

\begin{proof}
Following Stenzel \cite[Satz VII in Section 4]{Ste},
           the elementary divisors of the product $XY$ come in pairs.
  The generic point in $\mathcal V_{XY}$ is a diagonalizable matrix with nonzero eigenvalues of multiplicity 
  $2$. 
  We derive the dimension of the Stenzel variety $\mathcal V_{XY}$ from this description.
  Consider first the case where $n$ is even. 
  Then $\mathcal V_{XY}$ is the ${\rm GL}(n)$-orbit closure of the set of diagonal matrices ${\rm diag}(\lambda_1, \lambda_1, \ldots, \lambda_{n/2}, \lambda_{n/2})$.
  The group ${\rm GL}(n)$ has dimension $n^2$ and the stabilizer of the action, namely the set of invertible matrices with $2\times2$ blocks, has dimension $4n /2 = 2n$.
  Adding the $ n/ 2$ degrees of freedom for choosing the $\lambda_i$, 
  this set has dimension  $n^2 - 2n + n/2 - 1 = n(n-1) - \lfloor n/2 \rfloor - 1$ in projective space.
  For odd $n$, we take the orbit of ${\rm diag}(\lambda_1, \lambda_1, \ldots, \lambda_{n/2}, \lambda_{n/2}, 0)$.
  The stabilizer now has size $4(n-1)/2 + 1 = 2n-1$, so the projective dimension is $n^2 - (2n - 1) + (n-1)/2 - 1 = n(n-1) - \lfloor n/2 \rfloor - 1$.
  
  The dimension of $\mathcal V_{XY, \leq 2}$ is found in the same way.
  We take the ${\rm GL}(n)$-orbit closure of ${\rm diag}(\lambda_1, \lambda_1, 0, \ldots, 0)$.
  The dimension is $n^2 - 4 - (n-2)^2 + 1 - 1 = 4n - 8$. 
Finally, $\mathcal V_{XY}$ and $\mathcal V_{X^2Y^2}$ have the same dimension,
as do their restrictions to rank $2$. We omit the proof.
\end{proof}

\begin{example}[$n=3$]
 The Stenzel variety $\mathcal{V}_{XY}$ has dimension four, and it is the Segre variety $\PP^2 \times \PP^2 \subset \PP^8$. Its prime
   ideal is generated by the $2 \times 2$ minors of the matrix
 $$ 2 P - {\rm trace}(P) \cdot {\rm Id}_3 \,\, = \,\,\begin{bmatrix}
\,p_{11}-p_{22}-p_{33} &  2 p_{12} & 2 p_{13} \\
 2 p_{21} & -p_{11}+p_{22}-p_{33} & 2 p_{23} \\ 
 2 p_{31} &  2 p_{32} & -p_{11}-p_{22}+p_{33}\,
 \end{bmatrix}\!.
 $$
 The fourfold $\mathcal{V}_{X^2Y^2}$ has degree $20$ in $\PP^8$,
 with ideal  generated by $10$ cubics and $13$ quartics.
 \end{example}

\begin{example}[$n=4$]
The Stenzel variety $\mathcal{V}_{XY}$ has dimension $9$ and degree $14$ in $\PP^{15}$.
Its ideal is generated by $15$ quadrics, which look like the
$4 \times 4$ Pfaffians of a skew-symmetric $6 \times 6$ matrix of linear forms.
Its subvariety $\mathcal{V}_{XY,\leq 2}$ has dimension $8$
and degree $28$, with prime ideal generated by $16$ quadrics. Both varieties are
arithmetically Gorenstein.

The varieties $\mathcal{V}_{X^2Y^2}$ and $\mathcal{V}_{X^2 Y^2,\leq 2}$ have
the same dimensions as above, namely $9$ and $8$ respectively. The latter
is defined by $28$ cubics and $16$ quartics, and it has degree $196$.
\end{example}

The variety $\mathcal{V}_{XY^2X}$ arises because
the symmetric matrix $XY^2X$ has the same eigenvalues
as the non-symmetric matrix $X^2 Y^2$. These eigenvalues
 are real whenever $X$ and $Y$ are real. 
 
 \begin{example}
   Suppose $n \geq 5$.
   If $n$ is odd then
    $\mathcal{V}_{XY^2X}$ is the hypersurface defined by the symmetric determinant.   
  For even $n$, it is the ambient space $\PP({\rm Sym}^2\CC^n)$.
  The variety $\mathcal V_{XY^2X, \leq 2}$ 
  consists of $n\times n$ matrices of rank $\leq 2$; it has dimension $2n - 2$ and degree $\frac 1 2 \binom{2n - 2}{n-1}$.  
For $n = 4$,
the dimension is one less than expected.
Here, $\mathcal{V}_{XY^2X}$ is a hypersurface
in the  $\PP^9$ of symmetric $4 \times 4$ matrices.
Its defining polynomial has degree $6$ and $260$ terms. 
\end{example}

\section{Critical Equations}
\label{sec3}

Our main goal is to solve the optimization problem (\ref{eq:optproblem}).
We seek exact solutions,  to be reached by algebraic methods.
This rests on polynomial equations for the critical points.
We now introduce these
critical equations, both for implicit models and parametric models.
In this section, the data  $Q$ is a generic symmetric matrix, and the number of critical points is the ED degree of $\mathcal{M}$.
In Section~\ref{sec5}, the data will come from the projection Grassmannian.

The model $\mathcal{M}$ is an irreducible subvariety of ${\rm Gr}(2,n)$, which 
   is defined over $\RR$ and whose real points are Zariski dense.
   Its ideal  is generated by polynomials in the
   Pl\"ucker coordinates:
   \begin{equation}
   \label{eq:idealofmodel}
I_\mathcal{M}\, \, = \,\,
\langle f_1,f_2,\ldots,f_r \rangle \,+\,I_{{\rm Gr}(2,n)}\,\,\subset \, \,
\RR\bigl[ \, x_{ij}: 1 \!\leq\! i\! <\! j\! \leq\! n \,\bigr] .
\end{equation}   
Here $I_{{\rm Gr}(2,n)}$ is the ideal generated by the
      $\binom{n}{4}$  Pl\"ucker quadrics (\ref{eq:pluckerquadric}).
   We write $c = {\rm codim}(\mathcal{M})$.
    
Since the objective function factors through the squaring map
(\ref{eq:square}), we can also work with the
variety $\mathcal{M}^2$ in the projection Grassmannian $\overline{\rm p}{\rm Gr}(2,n)$.
In this setting, the prime ideal of the model is generated by polynomials
in the $\binom{n+1}{2}$ entries $p_{ij}$ of a symmetric $n \times n$ matrix.

In practice, many models $\mathcal{M}$ are unirational: they are given 
parametrically.
  Ideally, this parametrization
will be birational. In this case, we consider the critical equations
in the ring of polynomials in these parameters.
We illustrate this by  introducing our running example.

\begin{example}[Chow threefold of the twisted cubic]
Let $\mathcal{M} \subset {\rm Gr}(2,4)$ be parametrized by
\begin{equation}
\label{eq:chowtcc} A \, = \, \begin{bmatrix} 
\,1 & t_1 & t_1^2 & t_1^3\, \,\\
\,0 & 1 & t_2 & t_3\, \,\end{bmatrix}.  \quad
\end{equation}
The first row defines the twisted cubic curve in $\PP^3$,
and $\mathcal{M}$ is the threefold of all lines that
intersect this curve.
The Chow form is the  defining equation of $\mathcal{M}$ in Pl\"ucker coordinates: 
\begin{equation}
\label{eq:bezout} \qquad
{\rm det} 
\begin{small} \begin{bmatrix}
\, x_{12} & x_{13} & x_{23} \,\\
\, x_{13} & x_{14}+x_{23} & x_{24}\, \\
\, x_{23} & x_{24} & x_{34} \,
 \end{bmatrix} \,\, = \,\, 0.
 \end{small}
\end{equation}
The ideal of $\mathcal{M}^2$ is more complicated. It has $16$ minimal generators,
starting with the
$10$ quadrics that define $\overline{\rm p}{\rm Gr}(2,n)$ in $\PP( {\rm Sym}^2 \CC^4)$.
In addition, we have six more quartics, like
$$ \begin{footnotesize} \begin{matrix}
 2 p_{23}^2 p_{24} p_{33} - 2 p_{22} p_{24} p_{33}^2 - 2 p_{23}^3 p_{34} + 2 p_{14} p_{22} p_{24} p_{34}
 - 2 p_{13} p_{23} p_{24} p_{34} - 2 p_{12} p_{24}^2 p_{34} + 2 p_{22} p_{23} p_{33} p_{34} \\ + 2 p_{12} p_{24} p_{33} p_{34} 
 - 4 p_{13} p_{22} p_{34}^2 + 4 p_{12} p_{23} p_{34}^2
 - 2 p_{23}^2 p_{34}^2 + p_{11} p_{24} p_{34}^2 + p_{22} p_{24} p_{34}^2 - p_{11} p_{33} p_{34}^2 \\
 + p_{22} p_{33} p_{34}^2 - p_{24} p_{33} p_{34}^2 - p_{33}^2 p_{34}^2 + 2 p_{23} p_{34}^3 
 - 2 p_{34}^4 - 2 p_{14} p_{22} p_{23} p_{44} + 2 p_{13} p_{23}^2 p_{44}  
 + 2 p_{12} p_{23} p_{24} p_{44} \\ + 4 p_{13} p_{22} p_{33} p_{44}
 - 6 p_{12} p_{23} p_{33} p_{44} + 2 p_{23}^2 p_{33} p_{44} - p_{11} p_{24} p_{33} p_{44}
 - p_{22} p_{24} p_{33} p_{44} + p_{11} p_{33}^2 p_{44} \\
 - p_{22} p_{33}^2 p_{44} + p_{24} p_{33}^2 p_{44} + p_{33}^3 p_{44} - 2 p_{23} p_{33} p_{34} p_{44}
 + p_{24} p_{34}^2 p_{44} + 3 p_{33} p_{34}^2 p_{44} - p_{24} p_{33} p_{44}^2 - p_{33}^2 p_{44}^2.
\end{matrix} \end{footnotesize}
$$
For our optimization problem, it is better to use the cubic (\ref{eq:bezout}) or the parametrization (\ref{eq:chowtcc}).
\end{example}

While our objective function $||\tilde P - Q||^2$ is quadratic in the entries of $\tilde P$, the relations $\tilde P^2 = \tilde P$ and ${\rm trace}(\tilde P) = 2$ may be used to rewrite this function so that it is {\em linear} in $\tilde P$:
\begin{align*}
  {\rm trace}((\tilde P-Q)^2) \,=\, {\rm trace}(\tilde P^2) - 2\,{\rm trace}(\tilde PQ) + {\rm trace}(Q^2) 
  \,=\, 2 + {\rm trace}(Q^2) - 2\,{\rm trace}(\tilde PQ). 
\end{align*}
This allows us to replace \eqref{eq:optproblem} with an equivalent linear optimization problem
over the model:
\begin{equation}
\label{eq:optproblem2} 
  \textrm{Maximize} \quad {\rm trace}(\tilde PQ) \quad \textrm{subject to} \quad \tilde P \in \mathcal M^2 \,\subseteq \, {\rm pGr}(2,n). 
\end{equation}
If $\mathcal{M}^2$ equals  ${\rm pGr}(2,n)$,
then all $\binom{n}{2}$ critical points are real.
They are the projections onto the 2-dimensional $Q$-invariant subspaces of $\mathbb R^n$; see \cite[Theorem 3.3]{FH} and \cite[Section 4]{LLY}.

Written in Pl\"ucker coordinates, our objective function in (\ref{eq:optproblem2}) is the rational function
\begin{equation}
\label{eq:objectiveX} \phi_Q (X) \,\, := \,\,  2 \,\frac{{\rm trace}(X^2Q)}{{\rm trace}(X^2)}. 
\end{equation}
Let $\mathcal{J}_{\mathcal{M}}(X)$ denote the Jacobian matrix of the model $\mathcal{M}$.
This $\big(r \!+\! \binom{n}{4} \bigr)
 \times \binom{n}{2}$ matrix has rows indexed by
 the generators of $I_\mathcal{M}$,
  and columns indexed by the variables $x_{ij}$. 
The entries are the partial derivatives.
The rank of  $\mathcal{J}_{\mathcal{M}}(X)$ 
at a generic point $X \in \mathcal{M}$ is $c = {\rm codim}(\mathcal{M})$.

The singular locus of the model $\mathcal{M}$ is the subvariety 
of $ \PP(\wedge^2 \CC^n)$ defined by the ideal
$$ I_{{\rm Sing}(\mathcal{M})} \,\, := \,\,  I_\mathcal{M} \, \, + \,\, \bigl\langle \,
c \times c \,\, \hbox{minors of}\,\,\, \mathcal{J}_\mathcal{M}(X)\, \bigr\rangle. $$
The {\em augmented Jacobian matrix}  $\mathcal{J}^Q_\mathcal{M}(X)$
 is defined to be the $\bigl(r \!+\! \binom{n}{4} \!+\!2 \bigr) \times \binom{n}{2}$ matrix
\begin{equation}
\label{eq:augmentedJ}
  \mathcal J_{\mathcal M}^Q(X)\,\, = \,\,
  \left [\begin{array}{cccc}
    (XQ+QX)_{12} & (XQ+QX)_{13} & \cdots & (XQ+QX)_{n-1,n}\\
    x_{12} & x_{13} & \cdots & x_{n-1,n}\\
    \multicolumn{4}{c}{\mathcal J_{\mathcal M}(X)}
  \end{array}
  \right ].
\end{equation}
Up to scaling, the first row in (\ref{eq:augmentedJ}) is the gradient of the numerator ${\rm trace}(X^2Q)$ in
\eqref{eq:objectiveX}.
The second row is the gradient of the denominator ${\rm trace}(X^2)$.
For a model $\mathcal M$, the corresponding {\em Lagrangian ideal} \cite[Definition 2.5]{ranestad2025}
with respect to the data $Q$ is the following saturation:
\begin{equation}
\label{eq:lagrangianideal}
I^Q_{\mathcal M} \,\,= \,\,
 \biggl( I_\mathcal{M} \, \, + \,\,
\bigl\langle \,
(c\!+\!2) \times (c\!+\!2)\,\, \hbox{minors of}\,\,\, \mathcal{J}^Q_{\mathcal M}(X)\, \bigr\rangle \biggr) \,:\, I_{\rm Sing}(\mathcal M)^\infty. \qquad
\end{equation}
The model $\mathcal M$ is {\em ED-general} if it intersects the hypersurface $\{{\rm trace}(X^2)=0\}$ transversely. In this case, $I_{\mathcal M}^Q$ is also the critical ideal.
In any case,
the {\em critical ideal} is $I_{\mathcal M}^Q:\langle {\rm trace}(X^2) \rangle^\infty$.

\begin{proposition} \label{prop:EDideal}
For a general symmetric $n \times n$ matrix $Q$, the variety of the critical ideal is finite.
 It consists precisely of all non-singular points of $\mathcal{M}$
at which the objective function $\phi_Q(X)$ is critical. Their number is the ED degree of
the image under (\ref{eq:composition})
of the variety $\mathcal{M}$.
\end{proposition}

\begin{proof}
The proof is analogous to \cite[Lemma 2.1]{DHOST}. 
The key point is that $\phi_Q(X)$ and the entries in the first row of $\mathcal{J}^Q_{\mathcal{M}}(X)$
are linear  in the data $Q$. Thus,
for any non-singular point $X$ in $\mathcal{M}$, the set of data $Q$ 
for which the $(c+2) \times (c+2)$ minors of $\mathcal{J}^Q_{\mathcal{M}}(X)$ vanish
form an affine-linear subspace of dimension $c+n+1$. This ensures that the
incidence correspondence  in $\mathcal{M} \times \PP({\rm Sym}^2 \CC^n)$
defined by these minors has the correct dimension.
Its map onto the second factor is finite-to-one.
The saturation by $I_{{\rm Sing}(\mathcal{M})}$ is as in \cite[Lemma 2.1]{DHOST}.
We also saturate by ${\rm trace}(X^2)$ to remove possible
non-transverse intersections with the isotropic quadric.
The last sentence makes precise the sense in which we use ``ED degree'' in this paper.
\end{proof}

Models in the Grassmannian are never ED-general,
unless they are curves or surfaces:

\begin{theorem} \label{thm:EDgeneral}
Consider any subvariety $\mathcal M \subseteq {\rm Gr}(2,n)$.
    If ${\rm dim}(\mathcal M) \geq 3$, then $\mathcal M$ is not ED-general.
    If ${\rm dim}(\mathcal M) \leq 2$ and $\mathcal M$ is defined by general polynomials $f_i$, then $\mathcal M$ is ED-general.
\end{theorem}

\begin{proof}
The intersection ${\rm Gr}(2,n) \cap \{{\rm trace}(X^2) = 0\}$ is the ${\rm O}(n)$-orbit closure of
\begin{equation}\label{eq:nilpotent}
    \begin{bmatrix}
        0&  1 + i & 0\\
        -1 -i & 0 & -1+i\\
        0 & 1-i& 0
    \end{bmatrix}
    \oplus 
    \begin{bmatrix}
        0
    \end{bmatrix}
    \oplus \cdots \oplus \begin{bmatrix}
        0
    \end{bmatrix}.
\end{equation}
It is also the union of the ${\rm O}(n)$-orbits of \eqref{eq:base-matrix} and \eqref{eq:nilpotent}. 
It follows from \cite[Theorem 2']{KP} that the singular locus of this union is the ${\rm O}(n)$-orbit closure of \eqref{eq:base-matrix}.
Corollary \ref{cor:baselocus} states that this singular locus equals ${\rm Gr}(2,n) \cap \{ X^2 = 0\}$ and has codimension $3$
in ${\rm Gr}(2,n)$. 
If ${\rm dim}(\mathcal{M}) \geq 3$
then $\mathcal{M}$ intersects the base locus and
is hence tangent to the isotropic quadric.
If ${\rm dim}(\mathcal{M}) \leq 2$ then this
intersection is empty, unless the model $\mathcal{M}$ is very special.
\end{proof}

\begin{example}[$n=4$]
    The tangent curves in Proposition \ref{prop:tangentcurve} fail to be ED-general.
\end{example}

Theorem~\ref{thm:EDgeneral} allows us to
derive formulas for the ED degrees of generic curves and surfaces in 
the Grassmannian ${\rm Gr}(2,n)$. This will
be carried out in Theorem~\ref{thm:CI} for the case $n=4$.

We now return to our running example. This is a threefold
and hence not ED-general.

\begin{example}[Chow threefold of the twisted cubic]
Let $f$ denote the B\'ezout determinant in \eqref{eq:bezout}.
The augmented Jacobian of the threefold $\mathcal{M}$ it defines
in ${\rm Gr}(2,4)$ is the $4\! \times \!6$ matrix
  \begin{align*}
  \mathcal{J}^Q_{\mathcal{M}}(X)\,\,=\,\,
    \begin{bmatrix}
        \frac{\partial {\rm trace}(X^2Q)}{\partial x_{12}} & \frac{\partial {\rm trace}(X^2Q)}{\partial x_{13}} &\frac{\partial {\rm trace}(X^2Q)}{\partial x_{14}} & \frac{\partial {\rm trace}(X^2Q)}{\partial x_{23}}&\frac{\partial {\rm trace}(X^2Q)}{\partial x_{24}} & \frac{\partial {\rm trace}(X^2Q)}{\partial x_{34}}
        \smallskip \\     
    x_{12} & x_{13} & x_{14} & x_{23} & x_{24} & x_{34}\\
    x_{34} & \!\!\!-x_{24} & x_{23} & x_{14} & \!\!\!-x_{13} & x_{12} \smallskip \\
    \frac{\partial f}{ \partial x_{12}} & \frac{\partial f}{ \partial x_{13}} & \frac{\partial f}{ \partial x_{14}}&
    \frac{\partial f}{ \partial x_{23}} & \frac{\partial f}{ \partial x_{24}} & \frac{\partial f}{ \partial x_{34}}
    \end{bmatrix}.
  \end{align*}
  This matrix has linear entries in rows $1,2,$ and $3$, and quadratic entries in row $4$.
   Computation of the ideal (\ref{eq:lagrangianideal}), and saturation by ${\rm trace}(X^2)$,
   shows that this model  has ED degree~$42$.
\end{example}

The computation of critical equations is easier and faster when
the model $\mathcal{M}$ is given by a birational parametrization.
This usually takes the form of a $2 \times n$ matrix $A(t)$ which
depends on $s$ parameters $t = (t_1,t_2,\ldots,t_s)$,
and $t$ is identifiable from  the rowspan of $A(t)$.

The matrix $X= X(t)$ is now a function of $t$, and hence so is
$\phi_Q(X(t))$.
The aim is to compute critical points of the rational function $\phi_Q(X(t))$ with no constraints on $t$. 
We step back and compute the number of critical points of a general rational function 
in $s$ variables.

\begin{theorem} \label{thm:rtlfct}
    If $f(t_1, \ldots, t_s)$ and $\,g(t_1, \ldots, t_s)$ are general inhomogeneous polynomials of degrees
     $d$ and $e$, respectively, then the rational function $h(t) = f(t)/g(t)$ has 
        $\,\frac{  d(d-1)^s-e(e-1)^s}{d - e}$ 
      complex critical points if $d \neq e$, and it has $(s+1)(d-1)^s$ 
      complex critical points if $d = e$.
\end{theorem}

\begin{proof}
    The critical ideal is generated by the $2 \times 2$ minors of the matrix
    \begin{equation}
        \label{eq:fgmatrix}
        \begin{bmatrix}
       \, f & \frac{\partial f}{\partial t_1} &  \frac{\partial f}{\partial t_2}& \cdots & \frac{\partial f}{\partial t_s}\,
        \smallskip \\
       \, g & \frac{\partial g}{\partial t_1} &  \frac{\partial g}{\partial t_2}& \cdots & \frac{\partial g}{\partial t_s}\,\\
        \end{bmatrix}.
    \end{equation}
    For $d \neq e$, the Giambelli-Thom-Porteous formula 
    says that the degree of this ideal equals
     \begin{equation}\label{eq:GPT}
    \sum_{k = 0}^s (-1)^{s-k}\binom{s}{k}h_k(d,e).
    \end{equation}
    Here $h_k$ is the homogeneous symmetric function of degree $k$. 
    This simplifies to $\frac{e(e-1)^s \,-\, d(d-1)^s}{e \,-\, d}$.
    
    When $d = e$, a correction is needed because the matrix
    obtained from (\ref{eq:fgmatrix}) by taking the entries' leading forms
        drops rank at some points on the
    hyperplane at infinity in $\PP^{s-1}$. 
This follows from Euler's relation, which states that
$(-d,t_1,t_2,\ldots,t_s)^T$ is in the kernel of that matrix of
homogeneous polynomials.
Applying Giambelli-Thom-Porteous to the last $s$ columns
of that matrix, we see that the number of these solutions at
infinity is $s(d-1)^{s-1}$. We subtract this number
    from \eqref{eq:GPT} evaluated at $e = d$.
    The result simplifies to $(s + 1)(d-1)^s$. 
\end{proof}
\noindent 

We now return to our optimization problem,
where  $\mathcal{M}$ is parametrized by polynomials in $\RR[t_1,t_2,\ldots,t_s]$.
The critical ideal is generated by the $2 \times 2$ minors of the matrix
\begin{align*}
    \begin{bmatrix}
    {\rm trace}(X(t)^2 Q) &
        \frac{\partial {\rm trace}(X(t)^2Q)}{\partial t_1} & \frac{\partial {\rm trace}(X(t)^2Q)}{\partial t_2} & \cdots &  \frac{\partial {\rm trace}(X(t)^2Q)}{\partial t_s} \smallskip \\ 
        {\rm trace}(X(t)^2) &
        \frac{\partial {\rm trace}(X(t)^2)}{\partial t_1} & \frac{\partial {\rm trace}(X(t)^2)}{\partial t_2} & \cdots &  \frac{\partial {\rm trace}(X(t)^2)}{\partial t_s}
    \end{bmatrix}.
\end{align*}
By applying Theorem \ref{thm:rtlfct} to our situation, we obtain the following bound
on the ED degree.

\begin{corollary}\label{cor:param-ed}
    Suppose that the model $\mathcal M$ is parametrized by the $2 \times n$ matrix
$$
A(t) \,\, = \,\,
    \begin{bmatrix}
        a_{11}(t) & a_{12}(t) & \cdots & a_{1n}(t)\\
        a_{21}(t) & a_{22}(t) & \cdots & a_{2n}(t)
    \end{bmatrix}
    $$
    and that $t = (t_1,\ldots,t_s)$ is identifiable from the rowspan of $A(t)$. 
    If $d$ is the maximal degree among the $2 \times 2$ minors of $A(t)$, then the ED degree of $\mathcal M$ is at most $(s+1)(2d-1)^s$.
\end{corollary}

\begin{example}[$s\leq2,n=4$]
    It is possible for curves and surfaces to attain the bound in Corollary~\ref{cor:param-ed}.  
    Consider a parametric model $\mathcal{M}$ for lines in $3$-space, given by
\begin{align}\label{eq:gen-param}
    t \,\,\mapsto  \,\,
    \begin{bmatrix}
        1 & a_{12}(t) & a_{13}(t) & a_{14}(t)\\
        0 & 1 & a_{23}(t) & a_{24}(t)
    \end{bmatrix},
\end{align}
where $a_{12},a_{13},a_{14}$ are generic polynomials of degree $3$ and $a_{23},a_{24}$ are generic polynomials of degree $1$. The ED degree is $2\cdot 7 = 14$ for $s=1$ and $3 \cdot 7^2 = 147$ for $s=2$.
\end{example}

\begin{example}[$s=3, n=4$]
Consider the threefold defined by \eqref{eq:gen-param} 
where $a_{12},a_{23},a_{24}$ are linear,
$\deg(a_{13}) = 2$, and $\deg(a_{14}) = 3$. 
The rightmost $2 \times 2$ minor  has~the largest degree,
namely $4$, so the upper bound in Corollary~\ref{cor:param-ed} is $4 \cdot 7^3 = 1372$. 
If  the polynomials $a_{12}, \ldots,  a_{24}$ are generic,
then the ED degree of this model is $337$. 
For the Chow threefold of the twisted cubic, with parametrization
(\ref{eq:chowtcc}), the ED degree drops to $42$. Indeed, the bound in Corollary \ref{cor:param-ed} cannot be attained for models of dimension greater than $2$ by Theorem \ref{thm:EDgeneral}.
\end{example}

\section{The Grassmann Distance Degree}
\label{sec4}

Our primary objective is distance optimization inside the Grassmannian.
 The given data is a real
$2 \times n$ matrix $B$ of rank $2$. This represents a line in $\PP^{n-1}$.
From $B$, we compute a skew-symmetric matrix $Y $ as in (\ref{eq:lessambi}),
and we define $\tilde Q$ as the image of $Y$ under the map~(\ref{eq:composition}):
\begin{equation}
\label{eq:specialdata}
 \tilde Q \,\, = \,\, \frac{2}{{\rm trace}(Y^2)}\, Y^2 \qquad {\rm where} \qquad Y \,\, = \,\,B^T                                       
\begin{small} \begin{bmatrix} \,\,\,\,0 & 1\,\,\, \\ -1 & 0\,\,\, \end{bmatrix} \end{small} B.
\end{equation}
Thus, in this section, the matrix $\tilde Q$ is special. It is a
projection matrix in ${\rm pGr}(2,n)$.
As in the ED case, the critical points are 
given by a rank condition on $\mathcal J^{\tilde Q}_{\mathcal M}(X)$ as in \eqref{eq:augmentedJ}. 
However, for data $\tilde Q$ as in (\ref{eq:specialdata}), there are now extraneous 
critical points $X \in {\rm Gr}(2,n)$, where
 $\mathcal J^{\tilde Q}_{\mathcal M}(X)$ drops rank regardless of what the model is. 
These are critical points of the optimization~problem 
\begin{equation}\label{eq:groptproblem}
    \textrm{Maximize }\quad \frac{{\rm trace}(X^2\tilde Q)}{{\rm trace}(X^2)}
     \quad \textrm{subject to } \quad X \in {\rm Gr}(2,n).
\end{equation}
As in Section \ref{sec3}, they are defined by the constraint that the augmented Jacobian matrix $\mathcal J_{{\rm Gr}(2,n)}^{\tilde Q}(X)$ has rank at most ${\rm codim}({\rm Gr}(2,n)) + 1$.
  Here is a geometric  characterization:

\begin{proposition}\label{prop:excesscomponent}
    For general real data $\tilde Q$ in $ {\rm pGr}(2,n)$, the critical points of~\eqref{eq:groptproblem} are
    precisely the $\tilde Q$-invariant subspaces of dimension $2$ in $\mathbb R^n$.
    The dimension of the variety   
    \vspace{-0.4em}\begin{equation}
        \label{eq:extraideal}
        V\bigl(\,I^{\tilde Q}_{{\rm Gr}(2,n)} : \langle{\rm trace}(X^2)\rangle^{\infty} \,\bigr) \, \setminus \{Y\}
\end{equation}
    is $n-2$ for $n = 4,5$, and $2n - 8$ for $n \geq 6$. It is contained in the hypersurface 
    $V({\rm trace}(XY))$.
\end{proposition}

\begin{proof}
Let $\tilde P \in {\rm pGr}(2,n)$ be as in \eqref{eq:projmatrix}. Because ${\rm O}(n)$ acts transitively on ${\rm pGr}(2,n)$, we may assume that $\tilde P = \begin{footnotesize} \begin{bmatrix}
      {\rm Id}_2 & \! 0\\
      0 & \! 0
    \end{bmatrix} \end{footnotesize}$. The tangent vectors to ${\rm pGr}(2,n)$ at $\tilde P$ have the block structure $\dot{\tilde P} = 
     \begin{footnotesize}   \begin{bmatrix}
            0 & \! A^T \\
            A &  \! \! \! \! 0
        \end{bmatrix}\end{footnotesize}$
     for $(n-2)\times 2$ matrices $A$. A subspace $X$ is critical for \eqref{eq:groptproblem} if and only if ${\rm trace}(\dot{\tilde P}\tilde Q)=0$ for every $\dot{\tilde P}$. Decomposing $\tilde Q =
     \begin{footnotesize}   \begin{bmatrix}
            Q_{11} & \! Q_{12} \\
            Q_{12}^T & \! Q_{22}
        \end{bmatrix} \end{footnotesize}$,
         this is equivalent to  $Q_{12}=0$. Hence 
         the matrices $\tilde P$ and $\tilde Q$ commute.
    This implies  the first statement. 
    
    Let $Y = {\rm im}(\tilde Q)$ and $Y^{\perp}={\ker}(\tilde Q)$. By the first part of the proof, a critical 
    subspace~$X$ is spanned by vectors in $Y$ and  $Y^{\perp}$. The
    subspaces $X \subset Y^{\perp}$ form a Schubert variety of codimension $4$ in ${\rm Gr}(2,n)$. 
   The next case, when ${\rm dim}(X \cap Y)=1$ and ${\rm dim}(X \cap Y^{\perp})=1$,
    is a transverse intersection of two Schubert varieties of codimensions $n-3$ and $1$ respectively.
    This variety has codimension $n-2$ in ${\rm Gr}(2,n)$. The remaining case is  $X=Y$. The dimension     claim follows. Finally, ${\rm trace}(XY)$ vanishes if and only if  $X$ intersects $Y^{\perp}$
    non-trivially.
\end{proof}

The Zariski closure of the set in (\ref{eq:extraideal}) is called the {\em extraneous critical locus}.
We refer to its points as the {\em extraneous critical points}. 
This locus is independent of the  choice of
model $\mathcal{M}$. It consists of all critical points on the Grassmannian other than the data point.

Now, let us fix a model $\mathcal{M}$  and data $Y$, both in the Grassmannian ${\rm Gr}(2,n)$.
The intersection between  $\mathcal{M}$ and the extraneous critical locus
is contained in the variety defined by the ED critical ideal $I^{\tilde Q}_\mathcal{M}$.
That intersection can have a positive-dimensional component. Any such 
component would be contained in the {\em algebraic cut locus} of the point $Y$, which is
\begin{equation*}
 \mathcal C_Y \,\,= 
\,\,\bigl\{X \in \mathcal {\rm Gr}(2,n) : {\rm trace}(XY) = 0 \bigr\}.
\end{equation*}

The term {\em cut locus} comes from Riemannian geometry. It refers to
the set of points in a manifold
that are connected to a given point $Y$ by two or more length-minimizing geodesics;
 see \cite[Section 3]{LR}. The cut locus for the geodesic distance
 on ${\rm Gr}(2,n)$ is precisely the set of real points in $\mathcal{C}_Y$.
Indeed, for the geodesic distance problem in \cite{LR},
  a generic data point in ${\rm Gr}(2,n)$ can have infinitely many critical points 
  in $\mathcal{M}$ that lie in the cut locus. However, for smooth models, a (local) minimizer 
  cannot be among them. This is  \cite[Theorem 36]{LR}. 

A similar phenomenon occurs for the chordal distance (\ref{eq:frobenius2}). Proposition \ref{prop:excesscomponent} shows that the hypersurface $\mathcal C_Y$ can contain infinitely many critical points for \eqref{eq:optproblem}, even for generic data $Y \in {\rm Gr}(2,n)$.
 It is remarkable that, even though the metrics used in
 \cite{LR} and in this paper  enjoy very different properties, the same locus plays a crucial role in both optimization settings. For this reason we call the hypersurface $\mathcal C_Y \subset {\rm Gr(2,n)}$ the algebraic cut locus of $Y$.  
 
 \begin{definition} \label{def:GDdefs}
We define the {\em GD critical ideal} to be the saturation of the
Lagrangian ideal (\ref{eq:lagrangianideal}) by $\langle {\rm trace}(XY) \rangle^\infty$. If the model $\mathcal M$ is not ED-general, we also saturate by $\langle {\rm trace}(X^2)\rangle^{\infty}$; this will be always implicit in the following. By Proposition \ref{prop:excesscomponent}, saturating by $\langle {\rm trace}(XY)\rangle^{\infty}$ removes all extraneous critical points
from the variety defined by the ED critical ideal.
We define the {\em GD degree} of the model $\mathcal M$, denoted ${\rm GDdegree}(\mathcal M)$, 
to be the number of complex points in the variety defined by the GD critical ideal
 for generic data points $Y \in {\rm Gr(2,n)}$. 
\end{definition}

\begin{example}[Chow threefold of the twisted cubic]
The model in (\ref{eq:bezout}) has ED degree~$42$.
By taking data $\tilde Q$ in the Grassmannian as in \eqref{eq:specialdata}, we find that the GD degree is $10$.
\end{example}

\begin{remark}
We believe that the GD degree of every model $\mathcal{M}$ in
${\rm Gr}(2,n)$ is finite. If this holds, then
${\rm GDdegree}(\mathcal M) \leq {\rm EDdegree}(\mathcal{M})$ follows.
At present, we do not have a proof.
\end{remark}

  A model $\mathcal M$ in  ${\rm Gr}(2,n)$ is
     {\em GD-general} if it satisfies the following 
  for generic $Y \in {\rm Gr}(2,n)$:
    \begin{itemize}
  \item the model $\mathcal{M}$ intersects the extraneous critical locus transversely, and
  \item the extraneous critical locus contains the intersection of
the algebraic cut locus $\mathcal C_Y$ with  the variety defined by the ED critical ideal.
\end{itemize}
   The first condition implies that (\ref{eq:extraideal}) has
   the expected codimension in $\mathcal{M}$. The second 
   condition ensures that saturating by 
   $ {\rm trace}(XY)$ only removes extraneous critical points and not genuine ones. The model $\mathcal S_{25}$ from Example \ref{ex:spectralcurve} does not satisfy the second condition: it has a critical point in $\mathcal C_Y$ that is not extraneous in the sense of Proposition \ref{prop:excesscomponent}.

\smallskip

  If the model $\mathcal M$ is GD-general and small enough, then it will not intersect the extraneous critical locus. In this case, the GD degree of the model agrees with the ED degree. 

  \begin{example}[$n = 4$]\label{ex:n=4excess}
      For GD-general curves, the ED degree equals the GD degree. 
      For GD-general surfaces, the ED critical variety has  isolated points that are
      extraneous. Their number is the
      difference between ED degree and GD degree.
      For GD-general threefolds, the ED critical variety has extraneous curves.
      All of these extraneous points are critical for~\eqref{eq:groptproblem}.
  \end{example}
  
  \begin{example}[$n = 5$]
      For GD-general curves and surfaces in ${\rm Gr}(2,5)$,   the  ED degree equals
      the  GD degree. 
      For GD-general threefolds, the ED critical variety has isolated points that are extraneous.
      For GD-general fourfolds, the ED critical variety contains extraneous curves. 
      For GD-general fivefolds, the ED critical variety
      contains extraneous surfaces.
  \end{example}

The phenomena seen in the previous two examples have
the following general description.

\begin{theorem} \label{thm:excess}
Fix $n \geq 6$ and assume that the polynomials $f_1,\ldots,f_r$ in
  (\ref{eq:idealofmodel}) are generic.
If $r > 2n - 8$ then ${\rm EDdegree}(\mathcal M) = {\rm GDdegree}(\mathcal M)$. If $r \leq 2n-8$, 
then ${\rm EDdegree}(\mathcal M) > {\rm GDdegree}(\mathcal M)$ and the
 ED critical variety has extraneous components of dimension $2n - 8-r$. 
\end{theorem}

\begin{proof}
    By the genericity of its defining polynomials $f_1,\dots,f_r$,
     the model $\mathcal M$ is GD-general. The conclusion then follows from the 
     dimension statement in Proposition \ref{prop:excesscomponent}.
\end{proof}

An immediate consequence is that ${\rm GDdegree}(\mathcal M) < {\rm EDdegree}(\mathcal M)$ if ${\rm dim}(\mathcal M) \geq 4$. A discrepancy between the ED and GD degrees means
that the projection Grassmannian $\overline{\rm p}{\rm Gr}(2,n)$ 
is contained in the {\em ED discriminant} $\Sigma_{\mathcal{M}}$ of the given
model $\mathcal{M}$.
Recall from  \cite[Section~7]{DHOST} that $\Sigma_{\mathcal{M}}$
 is the hypersurface in $\PP({\rm Sym}^2 \CC^n)$ for which this drop occurs.
 In fact, based on computations, 
 the following stronger statement seems to hold for many models $\mathcal{M}$:
 \begin{equation}
\label{eq:inclusions5}
\mathcal{M} \,\subset \, \overline{\rm p}{\rm Gr}(2,n) \,\subset \,
\mathcal{V}_{X^2} \,\subset \, \Sigma_{\mathcal{M}} \, \subset \,
\PP( {\rm Sym}^2 \CC^n).
\end{equation}
We use the term {\em SD degree} for the number of
critical points when $Q$ is generic in $\mathcal{V}_{X^2}$.
The inclusion 
$\,\mathcal{V}_{X^2} \,\subset \, \Sigma_{\mathcal{M}} \,$
in (\ref{eq:inclusions5}) 
means that the SD degree is less than
the ED degree for $\mathcal{M}$.
 
In our running example, $\mathcal{M}$ has SD degree $20$. This is
 strictly between ${\rm GDdegree}(\mathcal{M}) = 10$
and ${\rm EDdegree}(\mathcal{M}) = 42$. 
 In fact, we observed that  ${\rm SDdegree}(\mathcal{M}) =
 2  \cdot {\rm GDdegree}(\mathcal{M}) $ whenever $\mathcal{M}$ is the
 Chow threefold of a toric curve in $\PP^3$; see Section~\ref{sec5}.
 For the Schubert varieties $\mathcal{S}_{ij}$ in Section~\ref{sec6},
 we conjecture that $\mathcal{V}_{X^2} \subset \Sigma_{\mathcal{S}_{ij}}$
if and only if $i=1$ or $(i,j) = (2,3)$.
These families of models will be introduced in slow motion in the remaining two sections.

\section{Lines in 3-Space}
\label{sec5}

This section is devoted to the GD problem for lines in $3$-space.
The model $\mathcal{M}$ is a subvariety of  ${\rm Gr}(2,4)$.
 If ${\rm dim}(\mathcal{M})=1$ then the union of
 its lines is a {\em ruled surface} in $\PP^3$. If ${\rm dim}(\mathcal{M})=2$
  then it represents a 
{\em line congruence}, i.e.~a surface of lines whose
union covers $\PP^3$. These are important in
computer vision \cite{PST}.
The term {\em line complex} is used classically for
a hypersurface in ${\rm Gr}(2,4)$. Such a threefold $\mathcal{M}$
is defined by one polynomial in the Pl\"ucker coordinates $x_{12},x_{13},x_{14},x_{23},x_{24},x_{34}$.
We begin with the case of generic complete intersections in ${\rm Gr}(2,4)$. 

\begin{theorem}\label{thm:CI}
 A surface in ${\rm Gr}(2,4)$ defined by generic polynomials 
 $f_1$ and $f_2$ of degrees $d_1$ and $ d_2$ has ED degree $2d_1d_2(d_1^2+d_2^2+d_1d_2+3)$.
  A curve in ${\rm Gr}(2,4)$ defined by three generic polynomials 
  $f_1,f_2,f_3$ of degrees $d_1, d_2, d_3$ has ED degree and GD degree $2d_1d_2d_3(d_1+d_2+d_3)$.
\end{theorem}

\begin{proof}
  These surfaces and curves are ED-general by Theorem~\ref{thm:EDgeneral}, so the number of critical points is the degree of the ideal $I_{\mathcal M}^Q$.
  By the Giambelli-Thom-Porteous formula, this degree~is
  \begin{equation*}
  \begin{matrix}
    {\rm deg}(I_{\mathcal M}^Q) \,\,= \,\,
    (2d_1\cdots d_r) \cdot \sum_{i_0 + \cdots + i_r = 4-r} \binom{i_0+2}{2} (d_1-1)^{i_1} \cdots (d_r-1)^{i_r}.
    \end{matrix}
  \end{equation*}
  One obtains the formulas in our assertion by plugging in $r = 2$ for surfaces and $r = 3$ for curves.
  By Example~\ref{ex:n=4excess}, the ED and GD degrees are equal for a general curve~in~${\rm Gr}(2,4)$.
\end{proof}

\begin{conjecture} \label{conj:surface}
    The GD degree of a surface defined by generic polynomials of degrees $d_1, d_2$ is $2d_1d_2(d_1^2 + d_2^2 + d_1d_2 + 2)$. This is $2d_1d_2$ less than the ED degree.
    A threefold defined by a generic polynomial $f$ of degree $d$ has ED degree $2d(d^3+3d+2)$ and GD degree $2d(d^3+2d)$.
\end{conjecture}

The following tables show these data  for surfaces and threefolds of low degree in ${\rm Gr}(2,4)$: 
\smallskip

\begin{footnotesize}
\!\!\!\!\!\!\!\!\!\!
\begin{tabular}{|c|cccc|}
  \toprule
  \diagbox[height=1.4em]{$\scriptstyle d_1$}{$ \scriptstyle d_2$} & 1 & 2  & 3 &  4 \\
  \midrule
  1 & 12 & 40 & 96 & 192 \\
  2 & & 120 & 264 & 496  \\
  3 & & &540 & 960 \\
  4 & & & &  1632\\
  \bottomrule 
  \multicolumn{5}{c}{ED degree of surfaces}
\end{tabular}
\hfill \,
\begin{tabular}{|c|cccc|}
  \toprule
  \diagbox[height=1.4em]{$\scriptstyle d_1$}{$ \scriptstyle d_2$} & 1 & 2  & 3 &  4 \\
  \midrule
  1 & 10 & 36 & 90 & 184 \\
  2 & & 112 & 252 & 480  \\
  3 & & &522 & 936 \\
  4 & & & &  1600\\
  \bottomrule
  \multicolumn{5}{c}{GD degree of surfaces}
\end{tabular}
\hfill \,
\begin{tabular}{|c|cc|}
  \toprule
  $d$ & ED degree & GD degree\\
  \midrule 
  1 & 12 & 6 \\
  2 & 64 & 48 \\
  3 & 228 & 198\\
  4 & 624 & 576\\
  5 & 1420 & 1350\\
  \bottomrule
  \multicolumn{3}{c}{Threefolds}
\end{tabular}
\end{footnotesize}

\smallskip 

Certain ruled surfaces,  line congruences and line complexes arise from
a given curve $\mathcal{C}$ of degree $d$ in $\PP^3$.
The {\em Chow threefold} of $\mathcal{C}$ is the set of lines
that intersect $\mathcal{C}$. 
Its defining Pl\"ucker polynomial has degree $d$.
If $d=1$ then the Chow threefold is the Schubert variety $\mathcal{S}_{13}$
in Example  \ref{ex:schubert4}.
For any given line $L$,
our GD problem asks for the line closest to $L$ among those
that intersect $\mathcal{C}$.
The {\em secant congruence} of $\mathcal{C}$ is the set of lines
that intersect $\mathcal{C}$ at two points. 
The best known ruled surfaces are quadrics in $\PP^3$. These
have two rulings of lines. Our GD problem
asks for the line in a ruling that is
closest to the data line $L$.~The~set~of lines tangent to $\mathcal{C}$
is the {\em tangent curve} in ${\rm Gr}(2,4)$. Its lines form the tangential surface of~$\mathcal{C}$.

We present
 an experimental study
for toric curves
$\mathcal{C}  =  \{\,
( 1 : t^{u_1} : t^{u_2} : t^{u_3}) \in \PP^3 \,: \, t \in \CC\, \}$.
The exponents are integers which satisfy $ 0 < u_1 < u_2 < u_3$ and
${\rm gcd}(u_1,u_2,u_3) = 1$.
We consider the three parametrically presented models
$\mathcal{M} \subset {\rm Gr}(2,4)$ that are derived from $\mathcal{C}$:
\begin{itemize}
\item The Chow threefold: \
$\, A(t,x,y) = \begin{bmatrix}
 1  &  t^{u_1}  &  t^{u_2} & t^{u_3} \\
 0 & 1 & x & y \end{bmatrix} $.
 \item The secant surface: \
$\, A(t,x) = \begin{bmatrix}
 1  &  t^{u_1}  &  t^{u_2} & t^{u_3} \\
 1  &  x^{u_1}  &  x^{u_2} & x^{u_3} \\
 \end{bmatrix} $.
\item The tangent curve: \
$\, A(t) = \begin{bmatrix}
 1  &  t^{u_1}  &  t^{u_2} & t^{u_3} \\
 0 & u_1 t^{u_1-1} &  u_2 t^{u_2-1} &  u_3 t^{u_3-1} 
 \end{bmatrix} $.
\end{itemize}

\noindent   These models have dimensions $3,2,1$.
   For the secant surface, the parametrization is two-to-one. 
   For the first and last, it is birational.
For the curve, we found the following result:

\begin{proposition}\label{prop:tangentcurve}
    The tangent curve of $\mathcal{C}$ has the same ED degree and GD degree. This degree equals
    $4 u_3+2 u_2-2u_1$, unless $u_3=u_1+u_2 $ and $u_3$ is even, in which case we subtract~$2$.
\end{proposition}   

\begin{proof}
    The numerator of the derivative of the objective function, denoted $f(t)$, has degree $4u_3 + 3u_2 + u_1 - 5$.
    The zeros of $f$ parametrize the critical points.
    This tangent curve is not ED general. 
    Indeed, $f$ has a factor of $t^{3u_1 + u_2 - 5}$. 
    The remaining $4u_3 + 2u_3 - 2u_1$ points are critical. 
    If $u_3 = u_1 + u_2$ and $u_3$ is even, then ${\rm trace}(X(t)^2)$ has a factor of $(t^2 + 1)^2$; therefore  $f$ has a factor of $t^2 + 1$.
    These points are extraneous because $t^2+1$ divides ${\rm trace}(X(t)^2)$.
\end{proof}

   The ED, SD, and GD degrees for some toric curves are shown in
   the following table:

\begin{footnotesize}
    \begin{center}
\begin{tabular}{c ccc c ccc c  ccc}
\toprule
& \multicolumn{3}{c}{Chow threefold}
& & \multicolumn{3}{c}{Secant surface}
& & \multicolumn{3}{c}{Tangent curve} \\
\cmidrule(lr){2-4}
\cmidrule(lr){6-8}
\cmidrule(lr){10-12}
$(u_1, u_2, u_3)$ 
& ED & SD& GD &
& ED & SD & GD &
& ED & SD & GD \\
\midrule
$(1, 2, 3)$ & 42 & 20 & 10 & & 19 & 15 & 15  & & 14 & 14 & 14 \\
$(1,2,4)$ & 55 & 26 & 13 & & 62 & 54 & 54  & & 18 & 18 & 18 \\
$(1,2,5)$ & 68 & 32 & 16 & & 111 & 98 & 98  & & 22 & 22 & 22 \\
$(1,2,6)$ & 81 & 38 & 19 & & 170 & 151 & 151 & & 26 & 26 & 26 \\
$(1,2,7)$ & 94 & 44 & 22 & & 235 & 209 & 209  & & 30 & 30 & 30 \\
$(1,3,4)$ & 54 & 24 & 12 & & 59 & 50 & 50  & & 18 & 18 & 18 \\
$(1,3,5)$ & 71 & 34 & 17 & & 122 & 108 & 108  & & 24 & 24 & 24 \\
$(1,3,6)$ & 84 & 40 & 20 & & 179 & 158 & 158  & & 28 & 28 & 28 \\
$(1,3,7)$ & 97 & 46 & 23 & & 262 & 234 & 234  & & 32 & 32 & 32 \\
$(1,4,5)$ & 74 & 36 & 18 & & 113 & 97 & 97 & & 26 & 26 & 26 \\
$(1,4,6)$ & 87 & 42 & 21 & & 212 & 189 & 189  & & 30 & 30 & 30 \\
\bottomrule
\end{tabular}
\end{center}
\end{footnotesize}

We also tested whether the optimal line intersects the data line.
For the examples above, this happens for the Chow threefolds, but not for the secant surfaces or tangent curves. 
Based on these computations, we venture the following conjectures
for every curve $\mathcal{C} \subset \PP^3$:

\begin{conjecture}
 The optimal line in the Chow threefold always intersects the data line.
   This does not hold for the secant surface and the tangent curve.
\end{conjecture}

\begin{conjecture} We have
$    {\rm SDdegree} = 2 \cdot {\rm GDdegree}$ for the Chow threefold. For the secant surface,
$    {\rm SDdegree} = {\rm GDdegree}$.
 \end{conjecture}
 
\section{Schubert Varieties}
\label{sec6}

The most prominent subvarieties of a Grassmannian are its Schubert varieties.
Their metric algebraic geometry was studied in \cite{LLY, LR}.
For $1 \leq i \!<\! j \leq n$, write $ \mathcal{L}_{ij} $ for the
linear space of skew-symmetric matrices $X = [x_{rs}]$ such that
$x_{rs} = 0$ unless $i \leq r$ and $j \leq s$. The {\em Schubert variety}
$\mathcal{S}_{ij} $ is the intersection~${\rm Gr}(2,n) \cap \mathcal{L}_{ij}$.
We use this term for any variety obtained from $\mathcal{S}_{ij}$ by an orthogonal transformation.
Note that \,$ {\rm dim}( \mathcal{S}_{ij} ) \, = \, 2n-i-j-1$.
The two extreme cases are
$\mathcal{S}_{12} = {\rm Gr}(2,n)$ and $\mathcal{S}_{n-1,n} = $ one point.
See also Example  \ref{ex:spectralcurve}.
We write $\mathcal{S}_{ij}^2 \subset \overline{\rm p}{\rm Gr}(2,n) \subset \PP({\rm Sym}^2\CC^n)$ for the image
of the Schubert variety 
$\mathcal{S}_{ij}$ under \eqref{eq:square}.

The implicit representation of the Schubert variety in  Pl\"ucker coordinates is the ideal
\begin{equation*}
\hbox{ideal of} \,\,\mathcal{S}_{ij} \,\,=\,\,
\hbox{ideal of} \,\,{\rm Gr}(2,n) \, + \,
\bigl\langle x_{rs} : r < i \,\,{\rm or}\,\, s < j \bigr\rangle.
\end{equation*}
We conjecture that the ideal of $\mathcal{S}_{ij}^2$ is generated
by quadrics and cubics. Conjecture \ref{conj:schubert} gives
the precise statement.
We write $P_{rs}$ for the $2 \times n$ submatrix of
$P$ given by the rows $r$ and~$s$.

\begin{conjecture} \label{conj:schubert}
The prime ideal of $\mathcal{S}_{ij}^2$ is 
generated by
the entries of $2 P^2 - {\rm trace}(P) \cdot P$, 
the $3 \times 3$ minors of $P$,
and the $2 \times 2$ minors  of the submatrices $P_{rs}$
where $ r < i$ or  $s < j $.
\end{conjecture}

This conjecture was verified for $n \leq 8$. 
We next present a census for $n = 4,5$.

\begin{example}[$n=4$]  \label{ex:schubert4} There are six Schubert varieties in ${\rm Gr}(2,4)$,
summarized as follows:
\[
\begin{footnotesize}
\begin{array}{|cccccccc|}
    \toprule
    & & \mathcal S_{12} & \mathcal S_{13} & \mathcal S_{23} & \mathcal S_{14} & \mathcal S_{24} & \mathcal S_{34}  \\
    \midrule 
{\rm dim} &&  4 & 3 & 2 & 2 & 1 &  0 \\
{\rm mingens} &&  (0,10) &(0,16) & (4,6) & (4,6) & (7,1) & (9,0) \\
{\rm degree} &&  12 & 12 & 4 & 4 & 2 &  1\\
{\rm ED\, degree} & & 6 & 10 & 3 & 3 & 2 & 1 \\
{\rm GD \,degree} && 1 & 2 & 1 & 1 & 2 & 1 \\
\bottomrule
\end{array}
\end{footnotesize}
\]
The rows labeled ``degree'' and ``mingens'' refer to the
variety $\mathcal{S}_{ij}^2$ in $\PP({\rm Sym}^2 \CC^4) = \PP^9$.
The ideal generators have degrees $1$ and $2$.
We list their numbers.
For instance, the  ideal of the surface $\mathcal{S}_{23}^2$ is generated by
$p_{11},p_{12},p_{13},p_{14}$ and the  entries of $2 {P_{-1}}^2 - {\rm trace}(P_{-1}) P_{-1}$,
where $P_{-i}$ is the
$3 \times 3$ matrix obtained from $P$ by deleting row $i$ and column $i$.
The ideal of the surface $\mathcal{S}_{14}^2$ is generated by
$p_{14},p_{24},p_{34}$ and
$p_{11}+p_{22}+p_{33}-p_{44}$ plus the
$2 \times 2$ minors of $P_{-4}$.
\end{example}

\begin{example}[$n=5$] \label{ex:schubert5} There are ten Schubert varieties in ${\rm Gr}(2,5)$,
summarized as follows:
$$ \begin{footnotesize}
\begin{array}{|ccccccccccc|}
\toprule 
 & \mathcal S_{12} & \mathcal S_{13} & \mathcal S_{23} &  \mathcal S_{14} & \mathcal S_{24} & \mathcal S_{15} & \mathcal S_{34} & \mathcal S_{25} & \mathcal S_{35} & \mathcal S_{45} \\
 \midrule 
{\rm dim} & 6 & 5 & 4 & 4 & 3 & 3 & 2 & 2 & 1 & 0 \\
{\rm mingens} & 
 (0,15) & \!(0,25) & (5,10) & (0,40) &  (5,16) & (5,20)  & (9,6) & (9,6)  &  \!(12,1) \!& \!(14,0) \\
 {\rm degree} &
40 & 40 & 12 &  28 & 12 &  8   & 4 & 4  & 2  & 1 \\
{\rm ED \,degree} &  10 & 24 & 6 & 16 & 10 & 4 & 3 & 3 & 2 & 1 \\
{\rm GD \,degree} & 1 & 2 & 1 & 2 & 2 & 1
& 1 & 2 & 2 & 1 \\
\bottomrule 
\end{array} 
\end{footnotesize}
$$
\end{example}
\newpage
The degree of $\mathcal S_{ij}^2$
depends only on the differences $n-j$ and $n-i$. Here are some values:
\begin{center}
\begin{footnotesize}
\begin{tabular}{|c|ccccccccc|}
  \toprule
  \diagbox[height=1.65em]{$\scriptstyle n-j$}{$\scriptstyle n-i$} & 1 & 2  & 3 &  4 &  5  &  6  &   7  &   8  & 9\\
  \midrule
  0 & 1 & 2 & 4 & 8 & 16 & 32 & 64 & 128 & 256\\
  1 & & 4 & 12 & 28 & 60 & 124 & 252 & 508 & 1020\\
  2 & & & 12 & 40 & 100 & 224 & 476 & 984 & 2004 \\
  3 & & & & 40 & 140 & 364 & 840 & 1824 & 3828\\
  4 & & & & & 140 & 504 & 1344 & 3168 & 6996\\
  5 & & & & & & 504 & 1848 & 5016 & 12\,012\\
  6 & & & & & & & 1848 & 6864 & 18\,876\\
  7 & & & & & & & & 6864 & 25\,740\\ 
  \bottomrule
\end{tabular}
\end{footnotesize}
\end{center}
When $n = j$, the variety $\mathcal S_{in}^2$ is the Veronese square $\nu_2(\PP^{n-i-1})$. 
This explains the powers of two in the $n -j = 0$ row. 
Rows 1-7 match the OEIS sequence {\tt A201385} up to a factor of 4.

\begin{conjecture} \label{conj:factor4}
  For $j < n$, the varieties $\mathcal S_{ij}^2$ satisfy
$\,    \deg(\mathcal S_{i,n-1}^2) = 4(2^{n-i-1} - 1) \,$ and
\begin{equation*}
 \quad
    \deg(\mathcal S_{ij}^2) \,\,=\,\, \deg(\mathcal S_{i+1, j}^2) + \deg(\mathcal S_{i,j+1}^2) \qquad
\hbox{  where $\,\,\deg(\mathcal S_{jj}^2) = 0$.} 
\end{equation*}
The solution to this recursion is
$  \deg(\mathcal S_{ij}^2) = \sum_{k = n-j+1}^{n-i-1} \binom{2n - i - j}{k} + 2\binom{2n - i- j - 1}{n - j}.$
\end{conjecture}

The ED degree was found in  \cite[Corollary 4.5]{LLY} 
for a special class of Schubert varieties:
$$ \begin{matrix}
  {\rm EDdegree}(\mathcal S_{i,i+1}) \,=\, \binom{n - i + 1}{2} \quad {\rm and} \quad
  {\rm EDdegree}(\mathcal S_{in}) \,=\, n - i.
  \end{matrix}
$$
For the general case we have the following conjecture:

\begin{conjecture} \label{conj:schubertED}
  If $i < j-1$, then the ED degree of the Schubert variety $\mathcal S_{ij}$ equals
$$    {\rm EDdegree}(\mathcal S_{ij}) \,\,=\,\, 2(n-j+1)^2(j - i) - (n - j + 1)(n-i).  $$
\end{conjecture}
\noindent
We verified Conjecture~\ref{conj:schubertED} for $n \leq 12$ with Gr\"obner basis computations over finite fields in {\tt Maple}. For $n = 13,14$ we verified the conjecture in {\tt HomotopyContinuation.jl}.

There is a dramatic drop from the ED degree to the GD degree for Schubert varieties.
\begin{theorem}\label{thm:schubertGD1}
The $n$ Schubert varieties  $\mathcal{S}_{12},\mathcal{S}_{23},
\ldots,\mathcal{S}_{n-1,n},\mathcal{S}_{1n}$
have GD degree $1$.
The minimizer is obtained from the data point by zeroing out appropriate Pl\"ucker coordinates. 
\end{theorem}

The proof of Theorem \ref{thm:schubertGD1} is given at the end of this section.
For all other Schubert varieties we put forth the following remarkable conjecture.
This was verified for $n \leq 12$.
 
\begin{conjecture}\label{conj:schubertGD}
If $1 < j - i < n-1$, then the Schubert variety $\mathcal S_{ij}$ has GD degree $2$.
\end{conjecture}

\begin{example}[Schubert curves]
Consider the curve $\mathcal{S}_{n-2,n}$ and write the data as
$$ B \,\,  = \,\, \begin{footnotesize} \begin{bmatrix} \,b_1 & b_2 & \cdots & b_{n-3} & 1 & \beta   & 0 \,\\
    \, c_1 & c_2 & \cdots & c_{n-3} & 0 &  \gamma & 1 \end{bmatrix} \end{footnotesize}
                                                  \,\,\, \hbox{and set} \,\,\,       \,                   
b^2 := \sum_{i=1}^{n-3} b_i^2   \,,\,\,
bc := \sum_{i=1}^{n-3} b_i c_i\,,\,\,
c^2 := \sum_{i=1}^{n-3} c_i^2 . $$
This model has GD degree $2$. A computation shows that
 the optimal solutions are
$$ A(t) \,\, = \,\, \begin{footnotesize} \begin{bmatrix} \, 0 & 0 & \cdots & 0 & 1 & t & 0 \\
0 & 0 & \cdots & 0 & 0 & 0 & 1 \end{bmatrix}, \end{footnotesize}
$$
where $\,
\bigl((-c^2\!-\!1) \beta + bc \gamma \bigr)\cdot t^2
   \,+\, \bigl((c^2\!+\!1)\beta ^2 - 2bc \beta \gamma  + b^2 \gamma^2  - c^2-1\bigr)\cdot t
   \,+\,  (c^2\!+\!1)\beta - bc \gamma      \,= \, 0$.
\end{example}

For fixed data $\tilde Q \in {\rm pGr}(2,n)$, define a map $\sigma_{\tilde Q} \,:\, \mathcal{M}_{\RR} \, \rightarrow \,\RR^2$ from the real locus of the model
that sends $X$ to the pair $(\lambda, \mu)$, where $\lambda \geq \mu$ are the nonzero eigenvalues of $\tilde P\tilde Q$ where $\tilde P$ is as in \eqref{eq:projmatrix}.
We call the image of $\sigma_{\tilde Q}$  the {\em spectral region} of the pair
$(\mathcal{M},\tilde Q)$; see Figure~\ref{fig:spectralregion}.

\begin{remark} \label{rmk:computespectral}
To compute the spectral region, we consider the variety in $\mathcal{M} \times \CC^2$ 
given by the ideal
$I_{\mathcal M} + \bigl\langle \det(\lambda\, {\rm Id}_n - \tilde P \tilde Q) ,\, \det(\mu\, {\rm Id}_n - \tilde P \tilde Q) \bigr\rangle$.
Consider the map from that variety onto the second factor $\CC^2$.
The branch locus of this map can be computed using elimination.
It is the algebraic boundary of the spectral region.
This is how the quartic in (\ref{eq:spectralcurve}) was computed.
\end{remark}

Our problem (\ref{eq:optproblem}) asks for the 
maximum coordinate sum $\lambda+ \mu$ over the spectral region.
For other metrics on ${\rm Gr}(2,n)$, like that  in \cite{LR}, other
monotone functions in $\lambda,\mu$ are maximized.
Multi-objective
distance optimization on  $\mathcal{M} \subset {\rm Gr}(2,n)$ means
identifying the Pareto frontier of the spectral region.
That frontier is a subset of the algebraic boundary~in Remark~\ref{rmk:computespectral}.
We say that the model $\mathcal{M}$ is {\em snug} if the Pareto frontier is a single point for all data $\tilde Q$.

We now give bounds on the spectral region and characterize the snug Schubert varieties. 
Fix
$E_j = {\rm diag}(0,0,\ldots,0,1,1,\ldots,1)$, where $j-1$ entries are zero.
 For an $n \times n$ matrix $M$ with real eigenvalues,
  we write the eigenvalues in decreasing order, i.e.,~$\lambda_1(M) \geq \cdots \geq \lambda_n(M)$.

\begin{theorem}\label{thm:snug}
If $\tilde P \in \mathcal S_{ij}$, then   $\lambda \!\leq \! \lambda_1(\tilde Q  E_i)$ 
and $\mu \!\leq\! \min(\lambda_2(\tilde Q  E_i), \lambda_1(\tilde Q E_j))$. 
Both bounds are attained for all $\tilde Q \in {\rm pGr}(2,n)$ if and only if
$\mathcal{S}_{ij}$ is snug if and only if $i = 1$ or $j - i = 1$.
  \end{theorem}

The last condition shows that among the $\binom{n}{2}$ Schubert
varieties,  only $2n-3$ are snug. To prove Theorem \ref{thm:snug} we need the lemma below, which follows from \cite[Theorem 1]{thompson}.

\begin{lemma}\label{lem:interlacing}
  Let $P_1 \in {\rm pGr}(k_1,n)$, $P_2 \in {\rm pGr}(k_2,n)$,
 and let $S \in {\rm pGr}(k, n)$ be the orthogonal projection onto the intersection of the linear spaces represented by $P_1$ and $P_2$. 
  Let $ Q \in {\rm Sym}^2 \RR^n$ be a projection matrix, i.e., $ Q^2 =  Q$.
  Then for $i = 1, \ldots, n - (k_2 - k)$ we have
    \begin{align*}
      \lambda_i( Q P_1)\, \geq \,\lambda_i( Q  S) \,\geq \,\lambda_{i + k_2 - k}( Q  P_2).
    \end{align*}
\end{lemma}

\begin{proof}[Proof of Theorem~\ref{thm:snug}]
  Taking $P_1 = E_i$, $P_2 = P$ and $Q = \tilde Q$ in Lemma~\ref{lem:interlacing} yields the inequalities
$\,  \lambda \leq \lambda_1(\tilde Q  E_i)\,$ and $\, \mu  \leq \lambda_2( \tilde Q  E_i)$.
  Taking instead $P_1 = E_j$ in Lemma~\ref{lem:interlacing} yields the inequality
$\,    \mu \leq \lambda_1(\tilde Q E_j) $.
This proves the first assertion.

We turn to the second assertion. First let $i \!=\! 1, j \!<\! n$, so
$\mathcal S_{1j} = \{ \tilde P : {\rm dim}(\tilde P \cap E_j) = 1 \}$.
  Let $\{q_1, q_2\}$ and $\{f_1,f_2, \ldots \}$ be orthonormal bases for the subspaces $\tilde Q$ and $E_j$
  such that $\langle q_1, f_1 \rangle = \sqrt{\lambda_1( \tilde Q E_j)}$  and
  $\langle q_2, f_2 \rangle = \sqrt{\lambda_2(\tilde Q E_j)}$.
  Let $\tilde P^*$ denote the orthogonal projection onto the span of $f_1$ and $q_2$,
  so $\lambda = \lambda_1( \tilde Q \tilde P^*)$ and $\mu = \lambda_2( \tilde Q  \tilde P^*)$.
  This implies $\lambda = 1 = \langle q_2, q_2 \rangle^2$ and
   $  \,\mu = \lambda_1(\tilde Q E_j) = \langle q_1, f_1 \rangle^2$,
   so the upper bounds are attained.

Next consider the Schubert variety $\mathcal S_{i,i+1} = 
\{\tilde P : \tilde P \subseteq E_i\}$.
  Let $\{q_1, q_2\}$ and $\{f_1, f_2,\ldots\}$ be orthonormal bases for $\tilde Q$ and $E_i$
  respectively such that $\langle q_1, f_1 \rangle = \sqrt{\lambda_1( \tilde Q E_i)}$ 
  and $\langle q_2, f_2 \rangle = \sqrt{\lambda_2(\tilde Q  E_i)}$.
    The desired minimizer $\tilde P^*$ is the orthogonal projection onto $\langle f_1, f_2 \rangle$.
    
It remains to be shown that none of the other  Schubert varieties are snug. Consider
 $$ \mathcal{S}_{ij} \, = \, \bigl\{\, \tilde P \,:\, {\rm dim}(\tilde P \cap E_j) = 1
 \,\, {\rm and} \,\, \tilde P \subseteq E_i \,\bigr\} \,\, = \,\, \mathcal{S}_{1j} \,\cap \, \mathcal{S}_{i,i+1}. $$
There exists a matrix $\tilde Q$ such that $\lambda_2(\tilde QE_i) < \lambda_1(\tilde Q E_j)$.
The inequalities
$\,  \lambda \leq \lambda_1(\tilde QE_i) \,$ and $\, \mu \leq \lambda_2(\tilde QE_i)\,$
are attained by the minimizer of \eqref{eq:optproblem} over $\mathcal S_{i,i+1} $,
which strictly contains $\mathcal S_{ij}$.
Generically, the minimizer intersects $E_j$ trivially, so it is not contained in $\mathcal S_{ij}$. 
By the uniqueness of the principal vectors, the bounds cannot be simultaneously achieved.
\end{proof}

\begin{proof}[Proof of Theorem~\ref{thm:schubertGD1}]
We prove the claim for $\mathcal S_{i,i+1}$. 
The argument is similar for $\mathcal S_{1n}$. Fix
\begin{equation*}
    B \,\,=\,\, \begin{bmatrix}
      \,b_{11} & \cdots & b_{1,i-1} &  1 & 0 & b_{1,i+2} & \cdots & b_{1n} \, \\
      \,b_{21} & \cdots & b_{2,i-1} & 0 & 1 & b_{2,i+2} & \cdots & b_{2n} \,
    \end{bmatrix}.
\end{equation*}
By the proof of Theorem~\ref{thm:snug}, the solution of \eqref{eq:optproblem} 
for $\mathcal{M} = \mathcal S_{i,i+1}$ is the orthogonal projection of the row span of $B$ onto $E_i$. 
This is found by setting $b_{1j} = b_{2j} = 0$ for $j = 1, \ldots, i-1$.
Hence the minimizer of \eqref{eq:optproblem} is a rational function of the data,
and so the GD degree is one. 
\end{proof}

\bigskip 
\bigskip

\noindent
\footnotesize {\bf Authors' addresses:}
\smallskip

\noindent Hannah Friedman, UC Berkeley, USA \hfill  {\tt hannahfriedman@berkeley.edu}
 
\noindent Andrea Rosana,  MPI MiS Leipzig, Germany \hfill {\tt andrea.rosana@mis.mpg.de}

\noindent Bernd Sturmfels, MPI MiS Leipzig, Germany \hfill {\tt bernd@mis.mpg.de}

\end{document}